\documentclass{amsart}
\usepackage{graphicx} 
\usepackage{amsfonts}
\usepackage{amssymb}
\usepackage{color}
\usepackage{amsmath}
\usepackage{amsthm}
\usepackage{multicol}
\usepackage{mathtools}
\usepackage{tikz,tikz-cd, tcolorbox}
\usepackage[colorlinks=true,linkcolor=blue,citecolor=blue]{hyperref}
\usepackage{setspace, mathrsfs}
\usepackage{bussproofs}
\theoremstyle{plain}
\newtheorem{theorem}{Theorem}[section]
\newtheorem{lemma}[theorem]{Lemma}
\newtheorem{proposition}[theorem]{Proposition}
\newtheorem{corollary}[theorem]{Corollary}
\usepackage{amsaddr}

\theoremstyle{definition}
\newtheorem{definition}[theorem]{Definition}
\newtheorem{remark}[theorem]{Remark}

\begin{document}
\title[Spectral duality for  De Morgan algebras]{Spectral duality for some modal and residuated groupoid expansions of De Morgan algebras}
\author{Joseph McDonald }
\address{Institute of Computer Science, Czech Academy of Sciences}

\email{mcdonald@cs.cas.cz}
\date{May 2026}
\maketitle

\begin{abstract} 
Stone demonstrated that the category $\mathbf{DLATT_{0,1}}$ of bounded distributive lattices is dually equivalent to the category  $\mathbf{Spec}$ of spectral spaces and Priestley showed that $\mathbf{DLatt_{0,1}}$ is dually equivalent to the category $\mathbf{Priest}$ of Priestley spaces so that $\mathbf{Spec}$ is equivalent $\mathbf{Priest}$. Cornish strengthened this by showing that $\mathbf{Spec}$ and $\mathbf{Priest}$ are in fact isomorphic. In this study, we investigate the duality theory of various lattice expansions of certain bounded distributive lattice-ordered algebras, known as \emph{De Morgan algebras}. In particular we obtain spectral duality results for the category $\mathbf{S4DM}$ of De Morgan algebras equipped with a closure operator, which we call \emph{S4 De Morgan algebras}, as well as for the category $\mathbf{DMGrp}$ of De Morgan groupoids. This is achieved by an appropriate adaptation of Bimb\'o's Priestley-style duality for general De Morgan algebras together with Urquhart's Priestley-style duality for relevance algebras under the isomorphism between $\mathbf{Priest}$ and $\mathbf{Spec}$.
\par
\vspace{.2cm}
\par
\vspace{.2cm}
\end{abstract}

\section{Introduction}
The duality theory of lattice-based algebras was pioneered by Stone \cite{stone1} which in its modern formulation (due to Doctor \cite{doctor}) states that every Boolean algebra is isomorphic to the clopen subsets of a compact zero-dimensional Hausdorff space, known as a \emph{Stone space}, and moreover, that the category $\mathbf{BA}$ of Boolean algebras and Boolean homomorphisms is dually equivalent to the category $\mathbf{Stone}$ of Stone spaces and continuous functions. This result was later generalized by Stone \cite{stone2} who showed that every bounded distributive lattice is isomorphic to the compact open subsets of a coherent sober compact $T_0$-space, known as a \emph{spectral space}, and moreover, that the category $\mathbf{DLatt_{0,1}}$ of bounded distributive lattices and bounded lattice homomorphisms is dually equivalent to the category $\mathbf{Spec}$ of spectral spaces and spectral maps (i.e., continuous functions whose preimage of a compact set is compact). Priestley \cite{priestly} gave an alternative duality for bounded distributive lattices by introducing certain partially ordered Stone spaces, known as \emph{Priestley spaces}, and demonstrated that every bounded distributive lattice is isomorphic to the clopen upsets of a Priestley space, and moreover, that $\mathbf{DLatt_{0,1}}$ is dually equivalent to the category $\mathbf{Priest}$ of Priestley spaces and continuous monotone functions. Since both $\textbf{Spec}$ and $\textbf{Priest}$ are dually equivalent to $\mathbf{DLatt_{0,1}}$, it follows that $\textbf{Spec}$ is equivalent to $\textbf{Priest}$. In fact a stronger claim is true: $\textbf{Spec}$ is isomorphic to $\textbf{Priest}$ (consult Cornish \cite{cornish} and Fleisher \cite{fleisher}). Duality theory, aside from being of much mathematical interest, has found many applications within theoretical computer science such as in denotational semantics of programs and program logics \cite{abram}, regular languages \cite{gehrke}, and probabilistic systems \cite{alv, jung1, jung2, jung3}.

In the first part of this paper, we study the duality theory of certain S4-type lattice expansions of De Morgan algebras by means of spectral spaces. A \emph{De Morgan algebra} is a bounded distributive lattice $A$ equipped with an additional operation $-\colon A\to A$ which is an order-inverting involution (equivalently, an involution satisfying De Morgan's identities with respect to meets and joins). De Morgan algebras play an important role within the model theory of non-classical logics as they provide algebraic models for 4-valued Belnap-Dunn logic \cite{belnap}. In particular, we study S4 De Morgan algebras which consist of a De Morgan algebra $A$ equipped with a closure operator $\nabla\colon A\to A$. This is not the first investigation of the duality theory of De Morgan algebras. Bimb\'o \cite{bimbo} established various duality results for De Morgan algebras from the perspective of De Morgan gaggle spaces, involution spaces, as well as product spaces. The De Morgan gaggle space approach arises as a special instance of the gaggle theory approach to the duality theory of bounded distributive lattices with operators developed by Bimb\'o and Dunn \cite{bimbo1}. The topological duals in this case are constructed by equipping a Priestley space with an additional binary relation satisfying certain conditions. The order-inverting involution on the induced bounded distributive lattice of clopen upsets of the Priestley space is defined through the additional binary relation which is determined by the distribution type of the operation of De Morgan complementation in a De Morgan algebra. The involution space approach involves equipping to a Priestley space, a single-variable function that is an order-inverting involution and then defining De Morgan complementation on the induced bounded distributive lattice of clopen upsets of the Priestley space through this function. This approach directly exploits the pleasant interaction that is known to exist between prime filters and the operation of De Morgan complementation on distributive lattices. The product space approach arises via Dunn's polarity semantics \cite{dunn} and involves starting with a compact topological space endowed with a product topology satisfying Priestley-style separation axioms with respect with both clopen upsets and clopen downsets. Here, the canonical frame is constructed from filter-ideal pairs and the canonical representation map sends every element into the collection of all filter-ideal pairs that intersect with that element.  

The duality established in this paper for S4 De Morgan algebras involves equipping to a spectral space $X$ an order-inverting involution $g\colon X\to X$ with respect to the specialization order on $X$, as well as binary relation $R$ on $X$ that is reflexive and transitive. By imposing additional conditions on the topology of $X$, we arrive at what we call an \emph{S4 De Morgan spectral space}. For an S4 De Morgan algebra $A$, it is shown that the prime filter spectrum $\mathfrak{P}(A)$ generated by the constructible topology, together with a single-variable function $g_A$ and binary relation $R_A$ on $\mathfrak{P}(A)$, gives rise to an S4 De Morgan spectral space. We then demonstrate that the algebra $\Omega(X)$ of compact open subsets of an S4 De Morgan spectral space $X$ form an S4 De Morgan algebra whose operation for De Morgan complementation $^*\colon\Omega(X)\to\Omega(X)$ is defined through $g$ and whose closure operator $\nabla_R\colon\Omega(X)\to\Omega(X)$ is defined through $R$. We then proceed to verify that every S4 De Morgan algebra $A$ is isomorphic to the algebra of compact open subsets $\Omega(\mathfrak{P}(A))$ on the prime filter spectrum $\mathfrak{P}(A)$ of $A$ and that every S4 De Morgan spectral space $X$ is homeomorphic and relationally isomorphic to the prime filter spectrum $\mathfrak{P}(\Omega(X))$ on the algebra of compact open subsets $\Omega(X)$ of $X$. With the introduction of suitable spectral frame morphisms on S4 De Morgan spectral spaces, we demonstrate that the category $\mathbf{S4DM}$ of S4 De Morgan algebras is dually equivalent to the category $\mathbf{S4DMSpec}$ of S4 De Morgan spectral spaces. The $R$-free reducts of S4 De Morgan spectral spaces, which we call \emph{De Morgan spectral spaces}, provide a duality between the category $\mathbf{DM}$ of general De organ algebras and $\mathbf{DMSpec}$ of De Morgan spectral spaces.

The obtained results are then used to study the duality theory of De Morgan groupoids, which arise as certain residuated lattice-ordered groupoid expansions of De Morgan algebras. In particular, a De Morgan groupoid is a De Morgan algebra $A$ equipped with a left-residuated pair of operations $(\cdot,\rightarrow)\colon A\times A\to A$ as well as a $0$-ary constant $t\in A$ that is a left-groupoid identity for the residual $\cdot$. In addition, the bottom universal bound $0\in A$ is required to be a unit element with respect to the residual $\cdot$ and the residual distributes over finite joins. De Morgan groupoids are closely related to the positive Ackermann groupoids studied by Routley and Meyer \cite{meyer} in their algebraic treatment of the relevance logic $\mathbf{B}^+$. In fact, every De Morgan groupoid may be viewed as the $\{-\}$-free reduct of a positive Ackermann groupoid. De Morgan groupoids are also closely related to the relevance algebras introduced by Urquhart \cite{urq}. De Morgan groupoids in this case are obtained by requiring that the dual lattice homomorphism acting on the underlying bounded distributive lattice reduct of a relevance algebra be an involution (i.e., a De Morgan complement). Urquhart studied the duality theory of relevance algebras by introducing certain Priestley spaces, equipped with an additional ternary relation satisfying certain conditions, as well as a single variable continuous decreasing function.

The duality established in this paper for De Morgan groupoids extends our spectral duality for De Morgan algebras by equipping a De Morgan spectral space $X$ with an additional ternary relation satisfying certain conditions along the lines of \cite{urq}. By imposing additional conditions on the topology of $X$, we arrive at what we call an \emph{De Morgan groupoid spectral spaces} (or \emph{DMGrp-spaces}). It is shown that for any De Morgan groupoid $A$, the prime filter spectrum $\mathfrak{P}(A)$ generated by the constructible topology gives rise to an DMGrp-space. We then demonstrate that the algebra $\Omega(X)$ of compact open subsets of a DMGrp-space $X$ forms an De Morgan groupoid. We then proceed to show that every De Morgan groupoid $A$ is isomorphic to the algebra of compact open subsets $\Omega(\mathfrak{P}(A))$ of the prime filter spectrum $\mathfrak{P}(A)$ of $A$ and that every DMGrp-space $X$ is homeomorphic and relationally isomorphic to the prime filter spectrum $\mathfrak{P}(\Omega(X))$ of the algebra of compact open subsets $\Omega(X)$ of $X$. With the introduction of suitable spectral frame morphisms on DMGrp-spaces, we demonstrate that the category $\mathbf{DMGrp}$ of De Morgan groupoids is dually equivalent to the category $\mathbf{DMGrpSpec}$ of DMGgrp-spaces. A simple extension of DMGrp-spaces by means of a binary reflexive transitive relation along the lines of our spectral duality for De Morgan algebras then gives rise to a dual equivalence between the category $\mathbf{S4DMGrp}$ of S4 De Morgan groupoids and the category $\mathbf{S4DMGrpSpec}$ of S4DMGrp-spaces.

\section{S4 De Morgan Algebras}

In this section, we describe some basic details of De Morgan lattices and De Morgan algebras. We then introduce the variety of S4 De Morgan algebras. 

\begin{definition}
    A \emph{De Morgan lattice} is a lattice $\langle A;\wedge,\vee\rangle$ equipped with an additional operator $-\colon A\to A$, known as a \emph{De Morgan complement}, satisfying:   
        \begin{enumerate}
            \item $a\leq b\Longrightarrow-b\leq-a$; 
            \item $--a=a$. 
        \end{enumerate}
\end{definition}
\begin{remark}
    A De Morgan lattice may be equivalently defined as a distributive lattice $\langle A;\wedge,\vee\rangle$ equipped with an involution $-\colon A\to A$ satisfying:
    \[-(a\wedge b)=-a\vee-b,\hspace{.4cm}-(a\vee b)=-a\wedge-b.\]
\end{remark}
Figure \ref{hasse diagrams of de morgan algebras} depicts Hasse diagrams of various De Morgan lattices.

\begin{figure}[ht]
    \centering
 \begin{tikzcd}
\bullet                     & \bullet                    &                             & \bullet                                         &                             \\
                            & \bullet \arrow[u, no head] & \bullet \arrow[ru, no head] &                                                 & \bullet \arrow[lu, no head] \\
\bullet \arrow[uu, no head] & \bullet \arrow[u, no head] &                             & \bullet \arrow[ru, no head] \arrow[lu, no head] &
\end{tikzcd}
\qquad
\begin{tikzcd}[row sep=0.28cm]
\bullet \\
\bullet \arrow[u, no head] \\
\bullet \arrow[u, no head] \\
\bullet \arrow[u, no head]
\end{tikzcd}
\par
\vspace{.5cm}
    \caption{The De Morgan lattices $\text{B}_2$, $\text{KL}_3$, $\text{DM}_4$, $\text{RegKL}_4$}
    \label{hasse diagrams of de morgan algebras}
\end{figure}
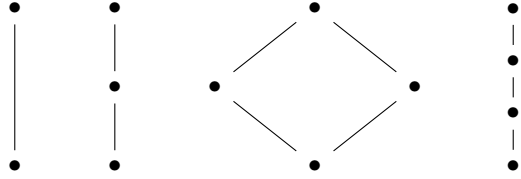
More concretely,  let $[0,1]=\{x\in\mathbb{R}:0\leq x\leq 1\}$ and define the function $\eta\colon[0,1]\to[0,1]$ by $\eta(x)=1-x$, then $\langle[0,1];\min,\max,\eta\rangle$ is a De Morgan lattice since $\langle[0,1];\min,\max\rangle$ is obviously a distributive lattice and $x\le y$ implies $\eta(y)\leq\eta(x)$ as well as $\eta(\eta(x))=x$ for all $x,y\in[0,1]$. 

 De Morgan lattices provide an algebraic model for Belnap-Dunn logic \cite{belnap} in the sense that a sequent $\Gamma\Rightarrow\Delta$ is a consequence of the sequents: \[\Gamma_1\Rightarrow\Delta_1\&\dots\&\hspace{.1cm}\Gamma_n\Rightarrow\Delta_n\] if and only if the quasi-equation: 
\[\bigwedge_{\gamma_1\in\Gamma_1}\gamma_1\leq\bigvee_{\delta_1\in\Delta_1}\delta_1\hspace{.1cm}\&\dots\&\bigwedge_{\gamma_n\in\Gamma_n}\gamma_n\leq\bigvee_{\delta_n\in\Delta_n}\delta_n\Longrightarrow\bigwedge_{\gamma\in\Gamma}\gamma\leq\bigvee_{\delta\in\Delta}\delta\] obtains in every De Morgan lattice. 
\begin{definition}
    Let $A$ be a De Morgan lattice. Then: 
    \begin{enumerate}
        \item $A$ is \emph{Boolean} if $a\wedge-a\leq b$; 
        \item $A$ is Kleene if $a\wedge-a\leq b\vee-b$; 
        \item $A$ is non-idempotent if $a=-a\Rightarrow a=b$; 
        \item $A$ is Kleene-regular if $a\leq -a\hspace{.1cm}\&-a\wedge b\leq a\vee-b\Rightarrow b\leq-b$. 
    \end{enumerate}
\end{definition}

\begin{proposition}[\cite{spencer}]
    A non-trivial De Morgan lattice $A$ is non-idempotent if and only if there exists a homomorphism from $A$ to $B_2$. 
\end{proposition}
\begin{proposition}[\cite{pynko}]
    If a De Morgan lattice $A$ is neither Kleene nor non-idempotent, then there exists an embedding from $\text{DM}_4$ into $A$. 
\end{proposition}
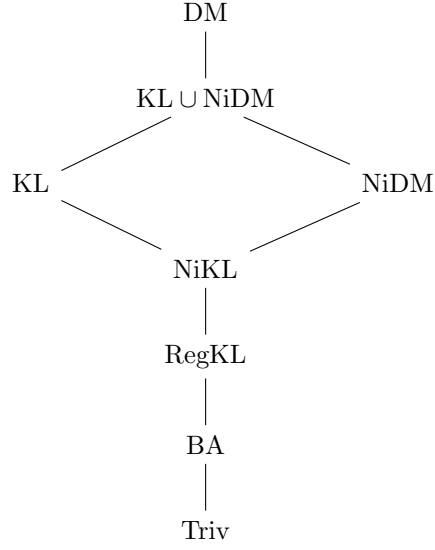
\begin{figure}[ht]
    \centering
\begin{tikzcd}
                              & \text{DM}                                          &                                  \\
                              & \text{KL}\cup\text{NiDM} \arrow[u, no head]        &                                  \\
\text{KL} \arrow[ru, no head] &                                                     & \text{NiDM} \arrow[lu, no head] \\
                              & \text{NiKL} \arrow[lu, no head] \arrow[ru, no head] &                                  \\
                              & \text{RegKL} \arrow[u, no head]                     &                                  \\
                              & \text{BA} \arrow[u, no head]                        &                                  \\
                              & \text{Triv} \arrow[u, no head]                   &                                 
\end{tikzcd}
  \caption{Lattice of quasivarieties of De Morgan lattices}
    \label{figure 1}
\end{figure}
Figure \ref{figure 1} depicts the lattice of quasivarieties of De Morgan lattices described by Pynko \cite{pynko} where:
\begin{itemize}
    \item $\text{Triv}$ is the trivial quasivariety; 
    \item $\text{BA}$ is the class of Boolean algebras (which is generated by $\text{B}_2$); 
    \item $\text{RegKL}$ is the class of regular Kleene lattices (which is generated by $\text{RegKL}_4$); 
    \item $\text{NiKL}$ is the class of non-idempotent Kleene algebras (which is generated by $\text{KL}_3\times\text{B}_2$); 
    \item $\text{KL}$ is the class of Kleene lattices (which is generated by $\text{KL}_3$); 
    \item $\text{NiDM}$ is the class of non-idempotent De Morgan lattices (which is generated by $\text{DM}_4\times\text{B}_2$); 
    \item $\text{KL}\cup\text{NiDM}$ is the quasivariety axiomatized by the following quasiequation $a=-a\Rightarrow b\wedge-b\leq c\vee-c$ and is generated by $(\text{KL}_3, \text{DM}_4\times\text{B}_2)$; 
    \item $\text{DM}$ is the variety of De Morgan lattices (which is generated by $\text{DM}_4$). 
\end{itemize}

 The variety of De Morgan algebras is obtained by adjoining lower and upper universal bounds to De Morgan lattices. In particular: 

\begin{definition}
    A \emph{De Morgan algebra} is an algebra $\langle A;\wedge,\vee,-,0,1\rangle$ such that: 
    \begin{enumerate}
        \item $\langle A;\wedge,\vee,-\rangle$ is a De Morgan lattice; 
        \item $\langle A;\wedge,\vee,0,1\rangle$ is a bounded lattice. 
    \end{enumerate}
\end{definition}
It is obvious that the examples provided for De Morgan lattices also provide examples of De Morgan algebras.

\begin{definition}
    An \emph{S4 De Morgan algebra} is a De Morgan algebra $A$ equipped with an additional operator $\nabla\colon A\to A$ satisfying: 
        \begin{enumerate}
            \item $\nabla(a\vee b)=\nabla a\vee\nabla b$; 
            \item $\nabla0=0$; 
            \item $a\leq\nabla a$; 
            \item $\nabla\nabla a\leq\nabla a$. 
        \end{enumerate}
\end{definition}

\section{Spectral duality for S4 De Morgan algebras}
We first obtain duality results for $\mathbf{S4DM}$ by means of spectral spaces by an appropriate adaptation of the Priestley-style duality developed by Bimb\'o \cite{bimbo} for $\mathbf{DM}$ under the isomorphism between $\mathbf{Priest}$ and $\mathbf{Spec}$. 
\subsection{S4 De Morgan Spectral Spaces}
If $X$ is a topological space, let ${K}(X)$ denote the collection of compact subsets of $X$, let  $O(X)$ denote the collection of open subsets of $X$, and let $\Omega(X)=K(X)\cap O(X)$.  

 Recall that a topological space $X$ is \emph{coherent} if $\Omega(X)$ is closed under finite intersections and $\Omega(X)$ forms a basis for $X$. Moreover, recall that $X$ is \emph{sober} if every completely prime filter in $\Omega(X)$ is of the form $\Omega_X(x)$ for some $x\in X$ where: \[\Omega_X(x)=\{U\in\Omega(X):x\in U\}.\] 
 
 In any topological space $X$, one can define a quasi-order $\leqslant$ on $X$, known as the \emph{specialization order} of $X$, by $x\leqslant y$ iff $x\in U$ implies $y\in U$ for all $U\in O(X)$. If $X$ is a $T_0$-space, then $\leqslant$ is in addition anti-symmetric and hence a partial order on $X$.
 \begin{definition}
     A topological space $X$ is a \emph{spectral space} if: 
     \begin{enumerate}
     \item $X$ is a compact space; 
         \item $X$ is a $T_0$-space; 
         \item $X$ is a coherent space; 
         \item $X$ is a sober space. 
     \end{enumerate}
 \end{definition}

 As each spectral space is a $T_0$-space, the specialization order of a spectral space is a partial order. If $X$ and $Y$ are spectral spaces, a map $f\colon X\to Y$ is a \emph{spectral map} if $f^{-1}[U]\in \Omega(X)$ for each $U\in \Omega(Y)$.

\begin{definition}\label{de morgan frame}
    An \emph{S4 De Morgan frame} is a quadruple $\langle X;\leq,g,R\rangle$ such that:
    \begin{enumerate}
        \item $\langle X;\leq\rangle$ is a poset; 
        \item $g\colon X\to X$ is an order-inverting involution; 
        \item $R\subseteq X\times X$ is reflexive and transitive.
    \end{enumerate}
\end{definition}
\noindent If $\langle X;\leq,g\rangle$ is a De Morgan frame, define $^*,\nabla_R\colon\wp(X)\to\wp(X)$ by:  
\[U^*=\{x\in X:g(x)\not\in U\}, \hspace{.2cm}\nabla_RU=R[U].\]

\begin{definition}\label{modal de morgan spectral space}
    An \emph{S4 De Morgan spectral space} is a relational topological space $\langle X;g,R,\tau\rangle$ satisfying the following conditions: 
    \begin{enumerate}
        \item $\langle X;\tau\rangle$ is a spectral space with specialization order $\leqslant$;   
        \item $\langle X;\leqslant,g\rangle$ is an S4 De Morgan frame;   
         \item if $U\in \Omega(X)$, then $U^*\in\Omega(X)$ and $\nabla_RU\in \Omega(X)$; 
         \item if $x\overline{R}y$, there exists $U\in\Omega(X)$ such that $x\in U$ and $y\not\in\nabla_RU$.    
    \end{enumerate}
\end{definition}
\begin{lemma}\label{separation property}
    If $X$ is an S4 De Morgan spectral space, then for all $x,y\in X$, if $x\not\leqslant y$, there exists $U\in\Omega(X)$ such that $x\in U$ and $y\not\in U$. 
\end{lemma}
\begin{proof}
    Assume that $x\not\leqslant y$. Then there exists some $U\in O(X)$ such that $x\in U$ but $y\not\in U$. Since $X$ is a coherent space, $\Omega(X)$ forms a basis for $X$ and hence: 
    \[U=\bigcup^n_{i=1}V_i\hspace{.2cm}\text{for $V_1,\dots, V_n\in\Omega(X)$}.\] Therefore we have that $x\in V_i$ for some $i\in\{1,\dots, n\}$ where $V_i\in\Omega(X)$ however $y\not\in V_i$ for each $i\in\{1,\dots,n\}$.  
\end{proof}
\subsection{Topological Representation of S4 De Morgan Algebras}  
Recall that for a bounded lattice $A$:
\begin{itemize}
    \item a non-zero element $a\in A$ is an \emph{atom} if there is no $b\in A$ such that $0<b<a$; 
    \item a non-top element $a\in A$ is a \emph{coatom} if there is no $b\in A$ such that $a<b<1$.
\end{itemize}
 We call a non-empty subset $x\subseteq A$ is a \emph{filter} provided:
\begin{itemize}
    \item $a\in x$ and $b\in x$ implies $a\wedge b\in x$; 
    \item $a\in x$ and $a\leq b$ implies $b\in x$.
\end{itemize}
 Moreover, if $x\subseteq A$ is a filter, then:
 \begin{itemize}
     \item $x$ is a \emph{proper filter} if $x\not=A$, i.e., $0\not\in x$; 
     \item $x$ is a \emph{prime filter} if $x$ is a proper and $a\vee b\in x$ implies $a\in x$ or $b\in x$. 
 \end{itemize}
 Dually, a non-empty subset $x\subseteq A$ is an \emph{ideal} provided: 
 \begin{itemize}
     \item $a\in x$ and $b\in x$ implies $a\vee b\in x$; 
     \item $a\in x$ and $b\leq a$ implies $b\in x$. 
 \end{itemize}
 Moreover, if $x\subseteq A$ is an ideal, then: 
 \begin{itemize}
     \item $x$ is a \emph{proper ideal} if $x\not=A$, i.e., $1\not\in x$; 
     \item $x$ is a \emph{prime ideal} if $x$ is proper and $a\wedge b$ implies $a\in x$ or $b\in x$. 
 \end{itemize}
 Now take any $a\in A$ and consider the following:
 \[{\uparrow}a=\{b\in A:a\leq b\},\hspace{.4cm}{\downarrow}a=\{b\in A:b\leq a\}\]
 It is easy to see that ${\uparrow}a$ is a filter, known as the \emph{principal filter} generated by $a$ and that ${\downarrow}a$ is an ideal, known as the \emph{principal ideal} generated by $a$. Clearly ${\uparrow}a$ (resp. ${\downarrow}a$) is a proper filter (resp. proper ideal) if $a\not=0$ (resp. $a\not=1$) and is moreover a prime filter (resp. prime ideal) if $a$ is an atom (resp. a coatom).   

 The following is well-known Prime Filter Theorem for distributive lattices and is essential for compactness in Lemma \ref{lemma 3.8} as well as for the construction of the isomorphism in the proof of Theorem \ref{rep theorem}.  
 \begin{theorem}
     Let $A$ be a districutive lattice, let $x\subseteq A$ be a filter, and let $y\subseteq A$ be an ideal such that $x\cap y=\emptyset$. Then there exists a prime filter $x_p\subseteq A$ such that $x\subseteq x_p$ with $y\cap x_p=\emptyset$. 
 \end{theorem}
\begin{definition}\label{prime spectrum}
    Let $A$ be an S4 De Morgan algebra. The \emph{prime spectrum} of $A$ is a relational topological space $\mathcal{S}_0(A)=\langle\mathfrak{P}(A); g_A,R_A,\tau(\beta)\rangle$ such that: 
    \begin{enumerate}
        \item $\mathfrak{P}(A)$ is the collection of all prime filters of $A$; 
        \item $g_A(x)=\{a\in A:-a\not\in x\}$; 
        \item $xR_Ay$ iff $a\in x$ implies $\nabla a\in y$; 
        \item $\tau(\beta)$ is the topology on $\mathfrak{P}(A)$ generated by the basis $\beta$ where: 
        \[\beta=\bigcup_{a\in A}\phi(a)\hspace{.2cm}\text{with}\hspace{.2cm}\phi(a)=\{x\in\mathfrak{P}(A):a\in x\}.\]
    \end{enumerate}
\end{definition}
\begin{remark}
    Note that in the Priestley-style topological representation of a distributive lattice $A$, one takes the prime filters $\mathfrak{P}(A)$ and forms a Priestley space by equipping to $\mathfrak{P}(A)$ the topology $\tau(\sigma)$ generated by the subbasis: 
    \[\sigma=\bigcup_{a\in A}\{\phi(a),\hspace{.1cm}\mathfrak{P}(A)\setminus\phi(a)\}.\] In the literature on spectral spaces, topologies of the form $\tau(\sigma)$ are often refereed to as the patch topologies whereas topologies of the form $\tau(\beta)$ are refereed to as constructible or spectral topologies (see \cite{dic} for more details). 
\end{remark}
\begin{lemma}\label{lemma 3.8}
    If $A$ is an S4 De Morgan algebra, then $\langle\mathfrak{P}(A);\tau(\beta)\rangle$ is a spectral space whose specialization order is given by set-theoretic inclusion.  
\end{lemma}
\begin{proof}
Since $A$ has a distributive lattice reduct, the  proof that $\langle\mathfrak{P}(A);\tau(\beta)\rangle$ is a spectral space follows from \cite{stone2} however we give explicit proofs that $\langle\mathfrak{P}(A);\tau(\beta)\rangle$ is a coherent sober space whose specialization order is given by set-theoretic inclusion. To see that $\langle\mathfrak{P}(A);\tau(\beta)\rangle$ is a coherent space, it remains to show that $\Omega(X_A)$ is closed under the formation of finite intersections. Therefore let $U=\bigcup_{i\in I}\phi(a_i)$ and $V=\bigcup_{i\in I}\phi(b_k)$ for $U,V\in \Omega(X_A)$. Then we have: 
  \begin{align*}
      U\cap V&=\bigcup_{i\in I,k\in K}(\phi(a_i)\cap\phi(b_k))
      \\&=\bigcup_{i\in I,k\in K}\{x\in\mathfrak{P}(A):a_i\in x\}\cap\{x\in\mathfrak{P}(A):b_k\in x\}
      \\&=\bigcup_{i\in I,k\in K}\{x\in\mathfrak{P}(A):a_i\in x\hspace{.1cm}\&\hspace{.1cm} b_k\in x\}
      \\&=\bigcup_{i\in I,k\in K}\{x\in\mathfrak{P}(A):a_i\wedge b_k\in x\}
      \\&=\bigcup_{i\in I,k\in K}\phi(a_i\wedge b_k).
  \end{align*}
Since $\beta=\bigcup_{a\in A}\phi(a)$ is a basis for $\mathcal{S}_0(A)$, it follows that  $\bigcup_{i\in I,k\in K}\phi(a_i\wedge b_k)=U\cap V\in\Omega(X_A)$ and thus $\mathcal{S}_0(A)$ is a coherent space.

    For sobriety, it suffices to show that every completely prime filter $x_p$ in $\Omega(X_A)$ is of the form $\Omega_{\mathfrak{P}(A)}(x)$ for some $x\in\mathfrak{P}(X_A)$ where $\Omega_{\mathfrak{P}(A)}(x)=\{U\in\Omega(X_A):x\in U\}$. Hence let $x$ be the filter generated by $\{a\in A:\phi(a)\in x_p\}$. We first verify that $x$ is a completely prime filter in $A$. Let $a,b\in x$, so that $\phi(a),\phi(b)\in x_p$. Since $x_p$ is a filter, it follows that $\phi(a)\cap\phi(b)\in x$ but we have already seen in the proof of coherence of $\langle\mathfrak{F}(A);\tau(\beta)\rangle$ that $\phi(a)\cap\phi(b)=\phi(a\wedge b)$ and hence $\phi(a\wedge b)\in x_p$, so $a\wedge b\in x$. Therefore $x$ is closed under finite meets. Now suppose $a\in x$ so that $\phi(a)\in x_p$ and observe that for any $y\in \phi(a)$, we have $a\in y$ so $b\in y$ since $a\leq b$ and hence $y\in\phi(b)$, hence $\phi(a)\subseteq\phi(b)$. Since $x_p$ is upwards closed, we have $\phi(b)\in x_p$ and thus $b\in x$ and hence $x$ is upward closed. To see that $x$ is a prime filter, assume $a\vee b\in x$ and hence $\phi(a\vee b)\in x_p$. However, we have: 
\begin{align*}
    \phi(a\vee b)&=\{x\in\mathfrak{P}(A):a\vee b\in x\}
    =\{x\in\mathfrak{P}(A):a\in x\hspace{.2cm}\text{or}\hspace{.2cm}b\in x\}
    \\&=\{x\in\mathfrak{P}(A):a\in x\}\cup\{x\in\mathfrak{P}(A):b\in x\}=\phi(a)\cup\phi(b).
\end{align*}
Therefore we have $\phi(a)\cup\phi(b)\in x_p$. Since $x_p$ is a prime filter, it follows that either $\phi(a)\in x_p$, in which case $a\in x$, or $\phi(b)\in x_p$, in which case $b\in x$. To see that $x$ is completely prime, suppose $\bigvee_{i\in I}a_i\in x$ so that $\phi(\bigvee_{i\in I}a_i)\in x_p$. Since $x_p$ is a completely prime filter, the construction of the topology $\tau(\beta)$ and the fact that $\phi$ is monotone and a homomorphism for finite joins yields: 
\[\phi\Biggl(\bigvee_{i\in I}a_i\Biggl)=\bigcup_{i\in I}\phi(a_i)\] and hence $\bigcup_{i\in I}\phi(a_i)\in x_p$. Since $x_p$ is completely prime, it follows that $\phi(a_i)\in x_p$ for some $i\in I$ and hence $a_i\in x$ for some $i\in I$. Hence we conclude that $x$ is a completely prime filter. We must now show that $x_p=\Omega_{\mathfrak{P}(A)}(x)$. The $\Omega_{\mathfrak{P}(A)}(x)\subseteq x_p$ inclusion is trivial by virtue of our construction of $x$. For the $x_p\subseteq\Omega_{\mathfrak{P}(A)}(x)$ inclusion, let $U=\bigcup_{i\in I}\phi(a_i)\in x_p$. Since $x_p$ is a completely prime filter, there exists $a_i\in A$ such that $\phi(a_i)\in x_p$ and hence $a_i\in x$ and thus $x\in\phi(a_i)$. Thus $\phi(a_i)\in\Omega_{\mathfrak{P}(A)}(x)$ so $U\in \Omega_{\mathfrak{P}(A)}(x)$. To see that $\subseteq$ is the specialization order of $X_A$, observe that since $X_A$ is a $T_0$-space, $x\not\subseteq y$ implies $x\not\leqslant y$. Now, if $x\subseteq y$, then for any basic open set $\phi(a)$ such that $x\in\phi(a)$, we have $a\in x$ so $a\in y$ and hence $y\in\phi(a)$ so $x\leqslant y$.   
\end{proof}

\begin{lemma}\label{lemma 3.9}
    If $A$ is an S4 De Morgan algebra, then $\langle\mathfrak{P}(A);\subseteq,R_A\rangle$ is an S4 De Morgan frame. 
\end{lemma}
\begin{proof}

We first verify that $g_A\colon\mathfrak{P}(X)\to\mathfrak{P}(A)$ is a well-defined function in the sense that $g_A(x)\in\mathfrak{P}(A)$ for any $x\in\mathfrak{P}(A)$. Hence let $x\in\mathfrak{P}(A)$ be arbitrary and suppose that $a,b\in g_A(x)$ so that $-a\not\in x$ and $-b\not\in x$. Now assume for the sake of contradiction that $a\wedge b\not\in g_A(x)$. Then $-(a\wedge b)\in x$ and since $-(a\wedge b)=-a\vee -b$, we have $-a\vee-b\in x$. Since $x$ is prime, either $-a\in x$ or $-b\in x$. Since either case contradicts our hypothesis that $-a\not\in x$ and $\-b\not\in x$, we conclude $-(a\wedge b)\not\in x$ so $a\wedge b\in g_A(x)$. Therefore $g_A(x)$ is closed under finite meets. To see that $g_A(x)$ is upward closed, take $a\in g_A(x)$ with $a\leq b$. Then $-a\not\in x$ and $-b\leq-a$ so we have $-b\not\in x$ which implies $b\in g_A(x)$. To see that $g_A(x)$ is prime, assume that $a\vee b\in g_A(x)$ so $-(a\vee b)=-a\wedge-b\not\in x$. Now assume for the sake of contradiction that $a\not\in g_A(x)$ and $b\not\in g_A(x)$. Then $-a\in x$ and $-b\in x$ and since $x$ is a filter, $-a\wedge-b\in x$, which contradicts our hypothesis. Thus either $a\in g_A(x)$ or $b\in g_A(x)$.  

     We now verify that $g_A$ is an order-inverting involution. Let $x,y\in\mathfrak{P}(A)$ such that $x\subseteq y$ and take $a\in g_A(y)$. Then $-a\not\in y$ and hence $-a\not\in x$ which implies $a\in g_A(a)$ so $g_A(y)\subseteq g_A(x)$. To see that $g_A$ is an involution, we have: 
\begin{align*}
    g_A(g_A(x))&=\{a\in A:-a\not\in g_A(x)\}=\{a\in A:--a\in x\}=\{a\in A:a\in x\}=x.
\end{align*}
     We now verify that $R_A$ is reflexive and transitive. For reflexivity, observe that $a\leq\nabla a$ for any $a\in A$ and hence for any $x\in\mathfrak{P}(A)$, if $a\in x$ then $\nabla a\in x$ since $x$ is upward closed, so $xR_Ax$ for any $x\in\mathfrak{P}(A)$. For transitivity, assume that $xR_Ay$ and $yR_Az$ and choose any $a\in x$. The former assumption yields $\nabla a\in y$ and the latter gives $\nabla\nabla a\in z$ but $\nabla\nabla a\leq\nabla a$ so $\nabla a\in z$ and hence $xR_Az$.  
\end{proof}
\begin{remark}
    As noted in \cite{bimbo}, if we assume that $x$ is merely a filter (i.e., not a prime filter), then $g_A(x)$ is not a filter. Moreover, if $\langle A;\wedge,\vee,-,0,1\rangle$ is a Boolean algebra, then $g_A$ defines the identity map on ultrafilters of $A$.    
\end{remark}
\begin{lemma}\label{lemma 3.11}
    If $A$ is an S4 De Morgan algebra, then its prime spectrum $\mathcal{S}_0(A)$ is an S4 De Morgan spectral space. 
\end{lemma}
\begin{proof}
    By Lemmas \ref{lemma 3.8} and \ref{lemma 3.9}, it remains to verify that conditions 3 and 4 of Definition \ref{modal de morgan spectral space} are satisfied. For condition 3, take any $a\in A$ and observe that: 
    \begin{align*}
\phi(a)^*&=\{x\in\mathfrak{P}(A):g_A(x)\not\in\phi(a)\}\\&=\{x\in\mathfrak{P}(A):a\not\in g_a(x)\}
\\&=\{x\in\mathfrak{P}(A):-a\in x\}\\&=\phi(-a)
\end{align*}
where $\phi(-a)\in\Omega(\mathcal{S}_0(A))$ so $\phi(a)^*\in\Omega(\mathcal{S}_0(A))$.

Now choose any $x\in\nabla_{R_S}\phi(a)$ so that $yR_Ax$ for some $y\in\phi(a)$. The latter gives $a\in y$ and which together with $yR_Ax$ implies $\nabla a\in x$ so $x\in\phi(a)$ and thus $\nabla_{R_A}\phi(a)\subseteq\phi(\nabla_A)$. Now assume that $x\in\phi(a)$ but $x\not\in \nabla_{R_A}\phi(a)$. The former gives $\nabla a\in x$ and the latter implies that for all $y\in\mathfrak{P}(A)$, either $y\not\in\phi(a)$ or $y\overline{R_A}x$. If $y\not\in\phi(a)$ for all $y\in\mathfrak{P}(A)$, then $a=0$ so $\nabla a=\nabla 0=0$ but $\nabla a\in x$ by hypothesis, and hence $0\in x$, which contradicts the fact that $x$ is proper. If on the other hand we have $y\overline{R_A}x$ for all $y\in\mathfrak{P}(A)$, then $x\overline{R_A}x$ but this contradicts the fact that $R$ is reflexive. Thus $x\in\nabla_{R_A}\phi(a)$ so $\phi(\nabla a)\subseteq \nabla_{R_A}\phi(a)$ and hence $\phi(\nabla a)=\nabla_{R_A}\phi(a)$. Hence $\nabla_{R_A}\phi(a)\in\Omega(\mathcal{S}_0(A))$ since $\phi(\nabla a)\in\Omega(\mathcal{S}_0(A))$.

For condition 4, assume that $x\overline{R_A}y$. Then there exists some $a\in A$ such that $a\in x$ but $\nabla a\not\in y$ so $x\in\phi(a)$ and $y\not\in\phi(\nabla a)=\nabla_{R_A}\phi(a)$ where where $\phi(a)\in\Omega(\mathcal{S}_0(A))$. Hence we conclude that $S_0(A)$ is an S4 De Morgan spectral space. 
\end{proof}

\begin{lemma}\label{space to algebra}
    If $X$ is an S4 De Morgan spectral space, then the induced algebra of compact open sets $\mathcal{A}_0(X)=\langle\Omega(X);\cap,\cup,^*,\emptyset,X,\nabla_R\rangle$ is an S4 De Morgan algebra.  
\end{lemma}
\begin{proof}
    Since $X$ is a spectral space, it follows by \cite{stone2} that $\langle\Omega(X);\cap,\cup,\emptyset,X\rangle$ is a bounded distributive lattice and since for any $U\in\Omega(X)$ we have $U^*\in\Omega(X)$ and $\nabla_{R}U\in\Omega(X)$ by Definition \ref{modal de morgan spectral space}(3), we know that $\mathcal{A}_0(X)$ is an algebra.

    To see that $^*\colon\Omega(X)\to\Omega(X)$ is an order-inverting involution, take $U,V\in\Omega(X)$ such that $U\subseteq V$ and assume $x\in V^*$. Then $g(x)\not\in V$ so $g(x)\not\in U$ which implies $x\in U^*$ so $V^*\subseteq U^*$. Now choose any $x\in U^{**}$ so that $g(x)\not\in U^*$. Then $g(g(x))\in U$  but $g(g(x))=x$ and hence $x\in U$ so $U^{**}\subseteq U$. Conversely, if $x\in U$, then $g(g(x))\in U$ and hence $g(x)\not\in U^*$ so $x\in U^{**}$, whence $U\subseteq U^{**}$ and thus $U=U^{**}$.

To see that $\nabla_R\colon\Omega(X)\to\Omega(X)$ is additive, observe that $x\in \nabla_R(U\cup V)$ iff there exists $y\in U\cup V$ with $yRx$ iff there exists $y\in U$ with $yRx$ or $y\in V$ with $yRx$ iff $x\in \nabla_RU\cup \nabla_RV$ and hence $\nabla_R(U\cup V)=\nabla_RU\cup\nabla_RV$. The proof that $\nabla_R$ is normal is trivial. To see that $\nabla_R$ is increasing, take $U\in\Omega(X)$ and choose any $x\in U$. Since $R$ is reflexive, we have $xRx$ and hence $x\in \nabla_RU$ so $U\subseteq\nabla_RU$. To see that $\nabla_R$ is idempotent, observe that $\nabla_RU\subseteq\nabla_R\nabla_RU$ follows from the fact that $\nabla_R$ is increasing and hence it remains to verify that $\nabla_R\nabla_RU\subseteq\nabla_RU$. If $x\in\nabla_R\nabla_RU$, then $yRx$ for some $y\in \nabla_RU$ where $zRy$ for some $z\in U$. Since $R$ is transitive, we have $zRx$ and since $z\in U$, we have $x\in\nabla_RU$.     
\end{proof}
\begin{theorem}\label{rep theorem}
    Every S4 De Morgan algebra $A$ is isomorphic to $\mathcal{A}_0(\mathcal{S}_0(A))$. 
\end{theorem}
\begin{proof}
    We show that $\phi(a)=\{x\in\mathfrak{P}(A):a\in x\}$ exhibits the desired isomorphism. We have already seen that $\phi(a)$ is compact open for each $a\in A$ and hence $\phi$ is well-defined. To see that $\phi$ is injective, take any $a,b\in A$ such that $a\not=b$. Suppose $a\not\leq b$ and consider the principal filter ${\uparrow}a$ and principal ideal ${\downarrow}b$. Since \[a\not\leq b\Longrightarrow{\uparrow}a\cap{\downarrow}b=\{c\in A:a\leq c\hspace{.1cm}\&\hspace{.1cm} c\leq b\}=\emptyset,\] it follows by the Prime Filter Theorem that one can find some prime filter $x_p\in\mathfrak{P}(A)$ such that ${\uparrow}a\subseteq x_p$ and $x_p\cap{\downarrow}b=\emptyset$ and hence $b\not\in x_p$. Thus $x_p\in\phi(a)$ and $x_p\not\in\phi(b)$ and hence $\phi(a)\not=\phi(b)$. We have already seen that each $U\in\Omega(\mathcal{S}_0(A))$ is of the form $\phi(a)$ for some $a\in A$ and hence $\phi$ is surjective. Thus $\phi$ is a bijection.  

    To verify that $\phi$ is a homomorphism, note that we have already demonstrated that $\phi(a\wedge b)=\phi(a)\cap\phi(b)$ and $\phi(a\vee b)=\phi(a)\cup\phi(b)$ in the proof of Lemma \ref{lemma 3.8}. Moreover, we have also verified $\phi(-a)=\phi(a)^*$ and $\phi(\nabla a)=\nabla_{R_A}\phi(a)$ in the proof of Lemma \ref{lemma 3.11}. Since $\mathfrak{P}(A)$ is the collection of all prime filters of $A$:   
\[\phi(1)=\{x\in\mathfrak{P}(A):1\in x\}=\mathfrak{P}(A),\hspace{.2cm}\phi(0)=\{x\in\mathfrak{P}(A):0\in X\}=\emptyset\] Hence we conclude that $\phi$ is an isomorphism from $A$ to $\mathcal{A}_0(\mathcal{S}_0(A))$. 
\end{proof}

\subsection{Algebraic Realization of S4 De Morgan Spectral Spaces}
In this subsection, we establish an algebraic realization theorem by demonstrating that every S4 De Morgan spectral space $X$ is homeomorphic and relationally isomorphic to the prime spectrum $\mathcal{S}_0(\mathcal{A}_0(X))$ of the S4 De Morgan algebra $\mathcal{A}_0(X)$ of compact open subsets of $X$. 
\begin{theorem}\label{homeo}
    Every S4 De Morgan spectral space $X$ is homeomorphic to $\mathcal{S}_0(\mathcal{A}_0(X))$. 
\end{theorem}
\begin{proof}
    We show that the function: \[\psi\colon X\to\mathcal{S}_0(\mathcal{A}_0(X));\hspace{.2cm}\psi(x)=\{U\in\Omega(X):x\in U\}\] exhibits the desired homeomorphism from $X$ to $\mathcal{S}_0(\mathcal{A}_0(X))$. We first verify that $\psi(x)$ is a prime filter in $\mathcal{A}_0(X)$ for each $x\in X$ to verify that $\psi$ is well-defined. Take any $U,V\in\Omega(X)$ such that $U,V\in \psi(x)$. Then $x\in U$ and $x\in V$ and hence $x\in U\cap V$. Since $X$ is a spectral space, $\Omega(X)$ is closed under finite intersections and hence $U\cap V\in \psi(x)$. Now take any $U,V\in\Omega(X)$ such that $U\in \psi(x)$ and $U\subseteq V$. Then $x\in U$ and hence $x\in V$ so $V\in \psi(x)$. Therefore $\psi(x)$ is upwards closed with respect to $\subseteq$ and is closed under $\cap$ and is thus a filter in $\Omega(X)$.

    Clearly $\psi(x)$ is a proper filter in $\Omega(X)$ in the sense that $\emptyset\not\in \psi(x)$ as $\emptyset$ is the bottom universal bound in $\Omega(X)$. To see that $\psi(x)$ is a prime filter, assume that $U\cup V\in \psi(x)$ so that $x\in U\cup V$. If $x\in U$, then $U\in \psi(x)$ and on the other hand, if $x\in V$, then $V\in \psi(x)$. Hence $\psi(x)$ is a prime filter in $\Omega(X)$ for all $x\in X$. 

    To see that $\psi$ is a bijection, we first show that $\psi$ is an injective function. Hence take any $x,y\in X$ such that $x\not=y$ and assume that $x\not\leqslant y$. By Lemma \ref{separation property}, there exists some $U\in\Omega(X)$ such that $x\in U$ and $y\not\in U$. Therefore $U\in\psi(x)$ and $U\not\in\psi(y)$ so $\psi(x)\not\subseteq\psi(y)$ and hence $\psi(x)\not=\psi(y)$. For surjectivity, since $X$ is a sober space, it follows that every completely prime filter, and hence every prime filter, in the lattice $\Omega(X)$ is of the following form:  
    \[\Omega_X(x)=\{U\in\Omega(X):x\in U\}=\psi(x)\] for some $x\in X$. To see that $\psi$ is a continuous function, since each basic open set in $\mathcal{S}_0(\mathcal{A}_0(X))$ is of the form $\phi(U)$ for some $U\in\Omega(X)$, we have: 
    \begin{align*}
        \psi^{-1}[U]&=\{x\in X:\Omega_X(x)\in\phi(U)\}\\&=\{x\in X:U\in\Omega_X(x)\}\\&=\{x\in X:x\in U\}\\&=U.
    \end{align*}
    To see that $\psi^{-1}$ is continuous, we have: 
\begin{align*}
    \psi[U]&=\{\Omega_X(x):x\in U\}\\&=\{\Omega_X(x):U\in\Omega_X(x)\}\\&=\phi(U).
\end{align*}
    As $\psi$ has been shown to a be a continuous bijection with continuous inverse, it follows that $\psi$ is a homeomorphism from $X$ to $\mathcal{S}_0(\mathcal{A}_0(X))$. 
\end{proof}

 Recall that if $X$ and $X'$ are sets equipped with binary relations $R\subseteq X\times X$ and $R'\subseteq X'\times X'$, then a function $f\colon X\to X'$ is a \emph{relational isomorphism} if:  
\[xRy\Longleftrightarrow f(x)R'f(y)\]
\begin{theorem}\label{relaitonal iso}
    The frame reduct $\langle X;R\rangle$ of any S4 De Morgan spectral space $X$ is relationally isomorphic to the frame reduct $\langle\mathfrak{P}(\Omega(X));R_{\Omega(X)}\rangle$ of $\mathcal{S}_0(\mathcal{A}_0(X))$.  
\end{theorem}
\begin{proof}
Observe that since $\psi(x)$ is a prime filter in $\mathcal{S}_0(\mathcal{A}_0(X))$ for any $x\in X$, we have that $\psi(x)R_{\Omega(X)}\psi(y)$ iff for all $U\in\Omega(X)$ we have $U\in \psi(x)$ implies $\nabla_{R}U\in\psi(y)$. Hence we have by Definition \ref{modal de morgan spectral space}(4) and the definition of $\psi$, we have that $x\overline{R}y$ iff there exists $U\in\Omega(X)$ such that $x\in U$ and $y\not\in\nabla_RU$ iff there exists $U\in\Omega(X)$ such that $U\in\psi(x)$ and $\nabla_RU\not\in\psi(y)$ iff $\psi(x)\overline{R_{\Omega(X)}}\psi(y)$.   
\end{proof}
\begin{theorem}\label{theorem 3.16}
    Any S4 De Morgan spectral space $X$ satisfies: \[\psi(g(x))=g_{\Omega(X)}(\psi(x))\] for all $x\in X$.   
\end{theorem}
\begin{proof}
   Note that for any S4 De Morgan spectral space $X$, since $\Omega(X)$ is an S4 De Morgan algebra and $\psi(x)$ is a prime filter in $\Omega(X)$, by Definition \ref{prime spectrum}(2) we have: 
   \[g_{\Omega(X)}(\psi(x))=\{U\in\Omega(X):U^*\not\in \psi(x)\}.\]
   The calculation therefore proceeds in the following manner: 
    \begin{align*}
        \psi(g(x))&=\{U\in\Omega(X):g(x)\in U\}
        \\&=\{U\in\Omega(X):x\not\in U^*\}
        \\&=\{U\in\Omega(X):U^*\not\in \psi(x)\}
        \\&=g_{\Omega(X)}(\psi(x)).
    \end{align*}
This completes the proof.
\end{proof}
\subsection{The Contravariant Functors $\mathcal{S}_F$ and $\mathcal{A}_F$}
In this subsection, we introduce the category $\mathbf{S4DMSpec}$ of S4 De Morgan spectral spaces and show that $\mathbf{S4DM}$ is dually equivalent to $\mathbf{S4DMSpec}$. 
\begin{definition}
    Let $A$ and $A'$ be S4 De Morgan algebras. A \emph{homomorphism} is a function $h\colon A\to A'$ satisfying the following conditions: 
    \begin{enumerate}
        \item $h(a\wedge b)=h(a)\wedge h(b)$; 
        \item $h(a\vee b)=h(a)\vee h(b)$; 
        \item $h(-a)=-h(a)$; 
        \item $h(0)=0'$; 
        \item $h(1)=1'$; 
        \item $h(\nabla a)=\nabla h(a)$. 
    \end{enumerate}
\end{definition}
\begin{definition}
    By $\mathbf{S4DM}$ we denote the category of S4 De Morgan algebras and homomorphisms. 
\end{definition}

The following spectral frame morphisms are a natural extension of the continuous frame morphisms between involution spaces that were introduced in \cite{bimbo}.

\begin{definition}\label{spectral frame morphism}
    Let $X$ and $X'$ be S4 De Morgan spectral spaces. Then a function $f\colon X\to X'$ is a \emph{spectral frame morphism} provided:
    \begin{enumerate}
        \item $f$ is a spectral map; 
        \item $f(g(x))=g(f(x))$; 
        \item $f^{-1}[\nabla_RU]=\nabla_Rf^{-1}[U]$.  
    \end{enumerate}
\end{definition}
\noindent The following is needed in order to justify Definition \ref{category of spec spaces}. 
\begin{proposition}\label{composition}
    Spectral frame morphisms on S4 De Morgan spectral spaces are closed under composition. 
\end{proposition}
\begin{proof}
 Let $\langle X_1;g_1,R_1,\tau_1\rangle$, $\langle X_2;g_2,R_2,\tau_2\rangle$, and $\langle X_3;g_3,R_3,\tau_3\rangle$  be S4 De Morgan spectral spaces and let $f_1\colon X_1\to X_2$ and $f_2\colon X_2\to X_3$ be spectral frame morphisms. The proof that $f_2\circ f_1\colon X_1\to X_3$ is a spectral map is a routine exercise. Take any $U\in\Omega(X_3)$ and observe that: \[(f_2\circ f_1)^{-1}[U]=f_1^{-1}\circ f_2^{-1}[U]=f_1^{-1}[f_2^{-1}[U]].\] Since $U\in\Omega(X_2)$ and $f_2$ is a spectral map, we have that $f_2^{-1}[U]\in\Omega(X_2)$ and since $f_1$ is a spectral map, we have $f_1^{-1}[f_2^{-1}[U]]\in\Omega(X_1)$. The calculation that $f_2\circ f_1(g(x))=g(f_2\circ f_1(x))$ directly exploits our assumptions that $f_1(g(x))=g(f_1(x))$ and $f_2(g(x))=g(f_2(x))$ and proceeds in the following manner:
 \[f_2\circ f_1(g(x))=f_2(f_1(g(x)))=f_2(g(f_1(x)))=g(f_2(f_1(x)))=g(f_2\circ f_1(x))\]
 Finally, to see that $f_2\circ f_1$ satisfies Definition \ref{spectral frame morphism}(3), observe that:
 \begin{align*}
     (f_2\circ f_1)^{-1}[\nabla_RU]&=f_1^{-1}\circ f_2^{-1}[\nabla_RU]\\&=f_1^{-1}[f^{-1}_2[\nabla_RU]]\\&=f^{-1}_1[\nabla_Rf^{-1}_2[U]]\\&=\nabla_Rf^{-1}_1[f^{-1}_2[U]]\\&=\nabla_Rf^{-1}_1\circ f_2^{-1}[U]\\&=\nabla_R(f_2\circ f_1)^{-1}[U]
 \end{align*}
 Hence we conclude that $f_2\circ f_1\colon X_1\to X_3$ is a spectral frame morphism
 whenever $f_1\colon X_1\to X_2$ and $f_2\colon X_2\to X_3$ are spectral frame morphisms. 
 \end{proof}
\begin{definition}\label{category of spec spaces}
    By $\mathbf{S4DMSpec}$ we denote the category of S4 De Morgan spectral spaces and their associated spectral frame morphisms. 
\end{definition}
\begin{lemma}\label{homo to spec}
    Let $A$ and $A'$ be S4 De Morgan algebras and let $h\colon A\to A'$ be a homomorphism. Then the map $\mathcal{S}_1\colon\mathcal{S}_0(A')\to\mathcal{S}_0(A)$ defined by $\mathcal{S}_1(h):=h^{-1}$ is a spectral frame morphism from the S4 De Morgan spectral space $\mathcal{S}_0(A')$ to the S4 De Morgan spectral space $\mathcal{S}_0(A)$.  
\end{lemma}
\begin{proof}
To see that $\mathcal{S}_1$ is a spectral frame morphism, we first verify that $\mathcal{S}_1(h)$ is a spectral map. Hence take any $U\in\Omega(\mathcal{S}_0(A))$ where $U=\phi(a)$ for some $a\in A$. It suffices to show that $\mathcal{S}_1(h)^{-1}[\phi(a)]\in\Omega(\mathcal{S}_0(A'))$. We have: 
\begin{align*}
    \mathcal{S}_1(h)^{-1}[\phi(a)]&=(h^{-1})^{-1}[\phi(a)]\\&=\{x\in\mathfrak{P}(A):h^{-1}[x]\in\phi(a)\}
    \\&=\{x\in\mathfrak{P}(A):a\in h^{-1}[x]\}
    \\&=\{x\in\mathfrak{P}(A):h(a)\in x\}
    \\&=\phi(h(a))
\end{align*}
where $\phi(h(a))\in\Omega(\mathcal{S}_0(A'))$. To see that Definition \ref{spectral frame morphism}(2) is satisfied, observe that:  
\begin{align*}
    \mathcal{S}_1(h)[g_A(x)]&=h^{-1}[g_A(x)]\\&=\{a\in A:h(a)\in g_A(x)\}
    \\&=\{a\in A:-h(a)\not\in x\}
    \\&=\{a\in A:h(-a)\not\in x\}
    \\&=\{a\in A:-a\not\in h^{-1}[x]\}
    \\&=g_A(h^{-1}[x])
    \\&=g_A(\mathcal{S}_1(h)[x])
\end{align*}
Finally, to see that Definition \ref{spectral frame morphism}(3) note that the previously established facts that $\phi(\nabla a)=\nabla_R\phi(a)$ and $(h^{-1})^{-1}[\phi(a)]=\phi(h(a))$ together with our assumption that $h(\nabla a)=\nabla h(a)$ for any $a\in A$ gives the following: 
\begin{align*}
    \mathcal{S}_1(h)^{-1}[\nabla_R\phi(a)]&=(h^{-1})^{-1}[\nabla_R\phi(a)]
    \\&=(h^{-1})^{-1}[\phi(\nabla a)]
    \\&=\phi(h(\nabla a))
    \\&=\phi(\nabla h(a))
    \\&=\nabla_R\phi(h(a))
    \\&=\nabla_R(h^{-1})^{-1}[\phi(a)]
    \\&=\nabla_R\mathcal{S}_1(h)^{-1}[\phi(a)]
\end{align*}
This completes the proof that $\mathcal{S}_1(h)$ is a spectral frame morphism from $\mathcal{S}_0(A')$ to $\mathcal{S}_0(A)$ whenever $h\colon A\to A'$ is a homomorphism.  
\end{proof}

\begin{lemma}\label{spec to homo}
    Let $X$ and $X'$ be S4 De Morgan spectral spaces and let $f\colon X\to X'$ be a spectral frame morphism. Then the map $\mathcal{A}_1\colon\mathcal{A}_0(X')\to\mathcal{A}_0(X)$ defined by $\mathcal{A}_1(f)=f^{-1}$ is a homomorphism from the S4 De Morgan algebra $\mathcal{A}_0(X')$ to the S4 De Morgan algebra $\mathcal{A}_0(X)$.  
\end{lemma}
\begin{proof}
    A standard set-theoretic argument shows that for every $U,V\in\Omega(X')$, we have $\mathcal{A}_1(f)[U\cap V]=\mathcal{A}_1(f)[U]\cap\mathcal{A}_1(f)[V]$, $\mathcal{A}_1(f)[U\cup V]=\mathcal{A}_1(f)[U]\cup\mathcal{A}_1(f)[V]$, $\mathcal{A}_1(f)[X]=X$, and $\mathcal{A}_1(f)[\emptyset]=\emptyset$. Hence $\mathcal{A}_1$ is a bounded lattice homomorphism from $\Omega(X')$ to $\Omega(X)$. To see that $\mathcal{A}_1$ is a homomorphism for $^*$, we have: 
    \begin{align*}
        \mathcal{A}_1(f)[U^*]=f^{-1}[U^*]&=\{x\in X:f(x)\in U^*\}
        \\&=\{x\in X:g(f(x))\not\in U\}
        \\&=\{x\in X:f(g(x))\not\in U\}
        \\&=\{x\in X:g(x)\not\in f^{-1}[U]\}
        \\&=f^{-1}[U]^*=\mathcal{A}_1(f)[U]^*
    \end{align*}
    The calculation that $\mathcal{A}_1$ is a homomorphism for $\nabla_R$ follows immediately from Definition \ref{spectral frame morphism}(3) since $\mathcal{A}_1(f)[\nabla_RU]=f^{-1}[\nabla_RU]=\nabla_Rf^{-1}[U]=\nabla_R\mathcal{A}_1(f)[U]$.                 
\end{proof}
Observe that Lemma \ref{lemma 3.11} together with Lemma \ref{homo to spec} provide a contravariant $\mathcal{S}_F=\langle\mathcal{S}_0,\mathcal{S}_1\rangle\colon\mathbf{S4DM}\to\mathbf{S4DMSpec}$ and that Lemma \ref{space to algebra} together with Lemma \ref{spec to homo} provide a contravariant functor $\mathcal{A}_F=\langle\mathcal{A}_0,\mathcal{A}_1\rangle\colon\mathbf{S4DMSpec}\to\mathbf{S4DM}$. In fact, more is true. 

\begin{lemma}\label{fully faithful 1}
    $\mathcal{S}_F=\langle \mathcal{S}_0,\mathcal{S}_1\rangle\colon\mathbf{S4DM}\to\mathbf{S4DMSpec}$ is fully faithful. 
\end{lemma}
\begin{proof}
  To see that $\mathcal{S}_F$ is faithful, let $A$ and $A'$ be S4 De Morgan algebras and let $h_1,h_2\colon A\to A'$ be homomorphisms. Assume $h_1\not= h_2$ so there exists $a\in A$ such that $h_1(a)\not= h_2(a)$. Without loss of generality, assume $h_1(a)\not\leq h_2(a)$. Then consider ${\uparrow}h_1(a)=\{b\in A':h_1(a)\leq b\}$ and ${\downarrow}h_2(a)=\{b\in A':b\leq h_2(a)\}$, which are the principal filter generated by $h_1(a)$ and the principal ideal generated by $h_2(a)$, respectively. Since $h_1(a)\not\leq h_2(a)$ implies ${\uparrow}h_1(a)\cap{\downarrow}h_2(a)=\{b\in A':h_1(a)\leq b\hspace{.1cm}\&\hspace{.1cm}b\leq h_2(a)\}=\emptyset$, the Prime Filter Theorem guarantees the existence of some $x_p\in\mathfrak{P}(A')$ such that ${\uparrow}h_1(a)\subseteq x_p$ with $h_2(a)\not\in x_p$. Hence $a\in h^{-1}_1[x_p]$ but $a\not\in h^{-1}_2[x_p]$ so $h^{-1}_1[x_p]\not=h^{-1}_2[x_p]$ and hence $\mathcal{S}_1(h_1)=h^{-1}_1\not=h^{-1}_2=\mathcal{S}_1(h_2)$.    

    To see that the faithful functor $\mathcal{S}_F$ is also a full functor, let $X$ and $X'$ be S4 De Morgan spectral spaces and let $f\colon X\to X'$ be a spectral frame morphism. By Lemma \ref{spec to homo}, there exists an S4 De Morgan algebra homomorphism $h\colon\mathcal{A}_0(X')\to\mathcal{A}_0(X)$ defined by $h=f^{-1}$ and by Lemma \ref{homo to spec}, there exists a spectral frame morphism $\mathcal{S}_1(h)\colon\mathcal{S}_0(\mathcal{A}_0(X))\to\mathcal{S}_0(\mathcal{A}_0(X'))$. Then by Theorem \ref{homeo} and Theorem \ref{relaitonal iso}, it follows that $X$ is relationally homeomorphic to $\mathcal{S}_0(\mathcal{A}_0(X))$ and $X'$ is relationally homeomorphic to $\mathcal{S}_0(\mathcal{A}_0(X'))$, which implies $\mathcal{S}_1(h)=f$.   
\end{proof}

\begin{lemma}\label{fully faithful 2}
    $\mathcal{A}_F=\langle \mathcal{A}_0,\mathcal{A}_1\rangle\colon\mathbf{S4DMSpec}\to\mathbf{S4DM}$ is fully faithful. 
\end{lemma}
\begin{proof}
    Let $X$ and $X'$ be S4 De Morgan spectral spaces and let $f_1,f_2\colon X\to X'$ be spectral frame morphisms. Assume $f_1\not=f_2$ so there exists $x\in X$ such that $f_1(x)\not= f_2(x)$. Without loss of generality, suppose $f_1(x)\not\leqslant f_2(x)$. Then by Lemma \ref{separation property}, there exists some $U\in\Omega(X')$ such that $f_1(x)\in U$ and $f_2(x)\not\in U$ and hence $U\in f_1(x)$ and $U\not\in f_2(x)$. This implies that $x\in f_1^{-1}[U]$ but $x\not\in f_2^{-1}[U]$ and thus $f_1^{-1}[U]\not\subseteq f_2^{-1}[U]$ so $\mathcal{A}_1(f_1)=f_1^{-1}\not= f_2^{-1}=\mathcal{A}_1(f_2)$. Hence $\mathcal{A}_F$ is faithful. 

    To see that the faithful functor $\mathcal{A}_F$ is full, let $A$ and $A'$ be S4 De Morgan algebras and let $h\colon A\to A'$ be a homomorphism. By Lemma \ref{homo to spec}, there exists a spectral frame morphism $f\colon\mathcal{S}_0(A')\to\mathcal{S}_0(A)$ defined by $f=h^{-1}$ and by Lemma \ref{spec to homo}, there exists an S4 De Morgan algebra homomorphism $\mathcal{A}_1(f)\colon\mathcal{A}_0(\mathcal{S}_0(A))\to\mathcal{A}_0(\mathcal{S}_0(A'))$. Then by Theorem \ref{rep theorem}, it follows that $A$ is isomorphic to $\mathcal{A}_0(\mathcal{S}_0(A))$ and $A'$ is isomorphic to $\mathcal{A}_0(\mathcal{S}_0(A'))$, which implies $\mathcal{A}_1(f)=h$. 
\end{proof}
\begin{theorem}\label{duality}
    $\mathbf{S4DM}$ is dually equivalent to $\mathbf{S4DMSpec}$
\end{theorem}
\begin{proof}
   In the proof of Lemma \ref{homo to spec} we have observed that $(h^{-1})^{-1}[\phi(a)]=\phi(h(a))$ so that $\mathcal{A}_1(\mathcal{S}_1(h))[\phi(a)]=\phi(h(a))$ for all $a\in A$. Hence it remains to verify that for every $x\in X$, we have $(f^{-1})^{-1}[\psi(x)]=\psi(f(x))$ so that $\mathcal{S}_1(\mathcal{A}_1(f))[\psi(x)]=\psi(f(x))$. The calculation proceeds in the following manner: 
\begin{align*}
    (f^{-1})^{-1}[\psi(x)]&=\{U\in\Omega(X):f^{-1}[U]\in\psi(x)\}
    \\&=\{U\in\Omega(X):x\in f^{-1}[U]\}
    \\&=\{U\in\Omega(X):f(x)\in U\}
    \\&=\psi(f(x)). 
\end{align*} 
Since $\phi$ is an isomorphism and $\psi$ is a relational homeomorphism, the results collected so far imply the existence of natural isomorphisms $\mathbf{id}_{\mathbf{S4DM}}\to\mathcal{A}_F\circ\mathcal{S}_F$ and $\mathbf{id}_{\mathbf{S4DMSpec}}\to\mathcal{S}_F\circ\mathcal{A}_F$ so that $\mathcal{S}_F=\langle \mathcal{S}_0,\mathcal{S}_1\rangle\colon\mathbf{S4DM}\to\mathbf{S4DMSpec}$ together with $\mathcal{A}_F=\langle \mathcal{A}_0,\mathcal{A}_1\rangle\colon\mathbf{S4DMSpec}\to\mathbf{S4DM}$ provide a dual equivalence between $\mathbf{S4DM}$ and $\mathbf{S4DMSpec}$.  
\end{proof}
If $\langle X;\leq,g,R\rangle$ is an S4 De Morgan frame, we call the reduct $\langle X;\leq,g\rangle$ a \emph{De Morgan frame} provided $\langle X;\leq\rangle$ is a poset and $g\colon X\to X$ is an order-inverting involution. An obvious restriction to Definition \ref{modal de morgan spectral space} the yields the spectral duals of general De Morgan algebras. 
\begin{definition}
    A \emph{De Morgan spectral space} is a topological space $\langle X;g,\tau\rangle$ satisfying the following conditions: 
    \begin{enumerate}
        \item $\langle X;\tau\rangle$ is a spectral space; 
        \item $\langle X;\leqslant,g\rangle$ is a De Morgan frame; 
        \item if $U\in\Omega(X)$, then $U^*\in\Omega(X)$. 
    \end{enumerate}
\end{definition}
\begin{definition}\label{spectral frame morphism}
    Let $X$ and $X'$ be De Morgan spectral spaces. Then a function $f\colon X\to X'$ is a \emph{spectral frame morphism} provided:
    \begin{enumerate}
        \item $f$ is a spectral map; 
        \item $f(g(x))=g(f(x))$.   
    \end{enumerate}
\end{definition}

Letting $\mathbf{DMSpec}$ denote the category of De Morgan spectral spaces and spectral frame morphisms, an obvious consequence of Theorem \ref{duality} is the following. 
\begin{theorem}\label{dem duality}
    $\mathbf{DM}$ is dually equivalent to $\mathbf{DMSpec}$. 
\end{theorem}
\begin{proof}
    If $A$ is a De Morgan algebra, Lemma \ref{lemma 3.11} provides a De Morgan spectral space $\mathcal{S}_0(A)=\langle\mathfrak{P}(A);g_A,\tau(\beta)\rangle$ and if $X$ is a De Morgan spectral space, Lemma \ref{space to algebra} induces a De Morgan algebra $\mathcal{A}_0(X)=\langle\Omega(X);\cap,\cup,^*,\emptyset,X\rangle$. By Theorems \ref{rep theorem} and \ref{homeo}, $A$ is isomorphic to $\mathcal{A}_0(\mathcal{S}_0(A))$ and $X$ is homeomorphic to $\mathcal{S}_0(\mathcal{A}_0(X))$. Lemmas \ref{fully faithful 1} and \ref{fully faithful 2} then guarantee the existence of fully faithful contravariant functors $\mathcal{S}_F=\langle\mathcal{S}_0,\mathcal{S}_1\rangle\colon\mathbf{DM}\to\mathbf{DMSpec}$ and $\mathcal{A}_F=\langle\mathcal{A}_0,\mathcal{A}_1\rangle\colon\mathbf{DMSpec}\to\mathbf{DM}$ with Theorem \ref{duality} providing the desired natural isomorphisms $\mathbf{id}_{\mathbf{DM}}\to\mathcal{A}_F\circ\mathcal{S}_F$ and $\mathbf{id}_{\mathbf{DMSpec}}\to\mathcal{S}_F\circ\mathcal{A}_F$ so that $\mathbf{DM}$ is dually equivalent to $\mathbf{DMSpec}$.    
\end{proof}
\section{Spectral Duality for De Morgan Groupoids}
In this section, we extend Theorem \ref{dem duality} to that of a dual equivalence between the category $\mathbf{DMGrp}$ of De Morgan groupoids and the category $\mathbf{DMGrpSpec}$ of De Morgan groupoid spectral spaces. 
\subsection{De Morgan Groupoids and DMGrp-Spaces}
By a \emph{unitial algebra} we mean a binar $\langle A;\cdot\rangle$ with a constant $0\in A$ satisfying $a\cdot 0=0\cdot a=0$ for all $a\in A$.    
\begin{definition}\label{dem groupoid}
A \emph{De Morgan groupoid} is an algebra $\langle A;\wedge,\vee,-,\cdot,\rightarrow,t,0,1\rangle$ satisfying the following conditions: 
\begin{enumerate}
    \item $\langle A;\wedge,\vee,-,0,1\rangle$ is a De Morgan algebra; 
    \item $\langle A;\cdot,t\rangle$ is a groupoid, i.e., $t\cdot a=a$; 
    \item $\langle A;\cdot,0\rangle$ is a unitial algebra; 
    \item $a\cdot(b\vee c)=(a\cdot b)\vee(a\cdot c)$; 
    \item $(a\vee b)\cdot c=(a\cdot c)\vee(b\cdot c)$; 
    \item $a\cdot b\leq c\Leftrightarrow a\leq b\rightarrow c$. 
\end{enumerate}
\end{definition}

\begin{remark}
 Observe that the $-$-free reduct $\langle A;\wedge,\vee,\circ,\rightarrow,t,0,1\rangle$ of a De Morgan groupoid corresponds to the bounded extensions of the positive Ackermann groupoids introduced by Routley and Meyer \cite{meyer}. Moreover, if $A$ is a De Morgan groupoid, then requiring that $-\colon A\to A$ merely be dual lattice homomorphism (i.e., an order inverting operation on $A$), then one arrives at the relevance algebras introduced by Urquhart \cite{urq}.  
\end{remark}
\begin{definition}\label{definition 4.3}
    If A is a De Morgan groupoid and $x,y\subseteq A$ are filters, let: 
    \[x\odot y=\{c\in A:a\cdot b\leq c\hspace{.2cm}\text{for some $a\in x$ and $b\in y$}\}.\]
\end{definition}
\begin{lemma}\label{lemma 4.3}
    Let $A$ be a De Morgan groupoid and let $\mathfrak{F}(A)$ denote the collection of all filters on $A$. Then the following conditions are satisfied: 
\begin{enumerate}
    \item if $x,y\in\mathfrak{F}(A)$, then $x\odot y\in\mathfrak{F}(A)$;  
    \item if $x,y\in\mathfrak{F}(A)$ and $z\in\mathfrak{P}(A)$ such that $x\odot y\subseteq z$, there exists $x',y'\in\mathfrak{P}(A)$ such that $x\subseteq x'$, $y\subseteq y'$, and $x'\odot y\subseteq z$. 
\end{enumerate}
 \end{lemma}
\begin{proof}
    The result follows from \cite[Lemma 2.1, Lemma 2.2]{urq}. 
\end{proof}

For the next definition, if $X$ is a set and $S\subseteq X^3$ is a ternary relation, let: 
\[U\circ V=\{z\in X:\exists xy(Sxyz\hspace{.1cm}\&\hspace{.1cm} x\in U\hspace{.1cm}\&\hspace{.1cm} y\in V)\}\]
\[U\rightarrow V=\{x\in X:\forall y((Sxyz\hspace{.1cm}\&\hspace{.1cm} y\in U)\Rightarrow z\in V)\}\]

\begin{definition}\label{demg space}
    A \emph{DMGrp-space} is a relational topological space of the following shape $\langle X;g,S,\Theta,\tau\rangle$ satisfying the following conditions: 
    \begin{enumerate}
        \item $\langle X;g,\tau\rangle$ is a De Morgan spectral space;  
        \item if $U,V\in\Omega(X)$, then $U\circ V\in\Omega(X)$ and $U\rightarrow V\in\Omega(X)$; 
        \item $Sxyz$ and $x'\leqslant x$ and $y'\leqslant y$ and $z\leqslant z'$ implies $Sx'y'z'$; 
        \item if $\overline{S}xyz$, $\exists U,V\in\Omega(X)(x\in U\hspace{.1cm}\&\hspace{.1cm} y\in V\hspace{.1cm}\&\hspace{.1cm} z\not\in U\circ V$); 
        \item $\Theta\in\Omega(X)$ and $\forall y,z(y\leqslant z\leftrightarrow\exists x(x\in \Theta\hspace{.1cm}\&\hspace{.1cm} Sxyz))$. 
        
    \end{enumerate}
\end{definition}

\subsection{Topological Representation of De Morgan Groupoids} In this subsection, we demonstrate that every De Morgan groupoid is isomorophic to the compact open subsets of a DMGrp-space. 
\begin{definition}\label{prime spectrum1}
    Let $A$ be a De Morgan groupoid. The \emph{prime spectrum} of $A$ is a relational topological space $\mathcal{F}_0(A)=\langle\mathfrak{P}(A); g_A,\Theta_A,S_A,\tau(\beta)\rangle$ such that: 
    \begin{enumerate}
        \item $\mathfrak{P}(A)$ is the collection of all prime filters of $A$; 
        \item $g_A(x)=\{a\in A:-a\not\in x\}$; 
        \item $\Theta_A=\{x\in\mathfrak{P}(A):t\in x\}$; 
        \item $S_Axyz\Longleftrightarrow x\odot y\subseteq z$; 
        \item $\tau(\beta)$ is the topology on $\mathfrak{P}(A)$ generated by the basis $\beta$ where: 
        \[\beta=\bigcup_{a\in A}\phi(a)\hspace{.2cm}\text{with}\hspace{.2cm}\phi(a)=\{x\in\mathfrak{P}(A):a\in x\}.\]
    \end{enumerate}
\end{definition}
\begin{lemma}\label{lemma 4.7}
    If $A$ is a De Morgan groupoid, then $\mathcal{S}_0(A)$ is a DMGrp-space. 
\end{lemma}
\begin{proof}
By Lemma \ref{lemma 3.11}, it follows that the reduct $\langle\mathfrak{P}(A);g_A,\tau(\beta)\rangle$ forms a De Morgan spectral space and hence it suffices to verify that conditions 2-5 of Definition \ref{demg space} are satisfied. For condition 2, assume $z\in\phi(a\cdot b)$, so that $a\cdot b\in z$. Now consider the principal filter ${\uparrow}a$ generated by $a$ and the principal filter ${\uparrow}b$ generated by $b$ so that ${\uparrow}a\odot{\uparrow}b\subseteq z$. By Lemma \ref{lemma 4.3}(2), there exists $x,y\in\mathfrak{P}(A)$ such that ${\uparrow}a\subseteq x$ and ${\uparrow}b\subseteq y$ where $a\in x$ and $b\in y$. Therefore we have $S_Axyz$ with $x\in\phi(a)$ and $y\in\phi(b)$ so that $z\in\phi(a)\circ\phi(b)$. Conversely, assume that $z\in\phi(a)\circ\phi(b)$ so that there exists $x,y\in\mathfrak{P}(A)$ such that $x\in\phi(a)$ and $y\in\phi(b)$ with $S_Axyz$. Then $x\odot y\subseteq z$ from which it can be shown that $a\cdot b\in z$, which implies $z\in\phi(a\cdot b)$. Hence we have $\phi(a\cdot b)=\phi(a)\circ\phi(b)$ with $\phi(a\cdot a)\in\Omega(\mathcal{S}_0(A))$ so $\phi(a)\circ\phi(b)\in\Omega(\mathcal{S}_0(A))$.

We now must demonstrate that $\phi(a\rightarrow b)=\phi(a)\rightarrow\phi(b)$. The left-to-right inclusion follows by definition and hence we verify the right-to-left inclusion. Hence assume that $a\rightarrow b\not\in x$ so that $b\not\in x\odot{\uparrow}a$ and let $z\in\mathfrak{P}(A)$ be such that $x\odot{\uparrow}a\subseteq z$ with $b\not\in z$. Observe that such a prime filter exists by the Prime filter Theorem. Then by Lemma \ref{lemma 4.3}(2), there exists some $y\in\mathfrak{P}(A)$ such that $a\in y$ and $S_Axyz$. Hence we conclude $\phi(a)\rightarrow\phi(b)\subseteq\phi(a\rightarrow b)$ and since $\phi(a\rightarrow b)\in\Omega(\mathcal{S}_0(A))$, we have $\phi(a)\rightarrow\phi(b)\in\Omega(\mathcal{S}_0(A))$, so condition 2 is satisfied. 

    For condition 3, assume that $S_Axyz$ with $x'\subseteq x$, $y'\subseteq y$, and $z\subseteq z'$. It suffices to show $x'\odot y'\subseteq z'$ so that $S_Ax'y'z'$. Hence take any $c\in x'\odot y'$ so that there exists $a\in x'$ and $b\in y'$ such that $a\cdot b\leq c$. Then $a\in x$ since $x'\subseteq x$ and $b\in y$ since $y'\subseteq y$. This, together with $a\cdot b\leq c$ implies $c\in x\odot y$ and thus $c\in z$ since $S_Axyz$. Then $c\in z'$ since $z\subseteq z'$ and therefore $x'\odot y'\subseteq z'$ which implies $S_Ax'y'z'$.  

    For condition 4, assume $\overline{S_A}xyz$ so that $x\odot y\not\subseteq z$. Then there exists some $c\in x\odot y$ such that $c\not\in z$ and hence there exists $a\in x$ and $b\in y$ such that $a\cdot b\leq c$. Hence $x\in\phi(a)$ and $y\in\phi(b)$ but $z\not\in \phi(a)\odot\phi(b)$ since $\overline{S_A}xyz$. For part 1 of condition 5, note that $\Theta_A=\{x\in\mathfrak{P}(A):t\in x\}=\phi(t)$ with $\phi(t)\in\Omega(\mathcal{S}_0(A))$. Part 2 of condition 5 follows from the fact that $t\cdot a=a$ for all $a\in A$.    
\end{proof}
\begin{lemma}\label{lemma 4.8}
    If $X$ is a DMGrp-space, then $\mathcal{G}_0(X)=\langle\Omega(X);\cap,\cup,^{*},\circ,\rightarrow,\Theta,\emptyset,X\rangle$ is a De Morgan groupoid. 
\end{lemma}
\begin{proof}
   By Lemma \ref{space to algebra}, it follows that $\langle\Omega(X);\cap,\cup,^*,\emptyset,X\rangle$ is a De Morgan algebra and hence it remains to verify that conditions 2-6 of Definition \ref{dem groupoid} are satisfied. First observe that condition 5 of Definition \ref{demg space} guarantees that $\Theta\in\Omega(X)$. It is trivial that $\langle\Omega(X);\circ,\Theta\rangle$ is a unitial algebra since as clearly $U\circ\emptyset=\emptyset\circ U=\emptyset$ for every $U\in\Omega(X)$. To verify that $U\circ V\subseteq W$ iff $U\subseteq V\rightarrow W$, we first assume that $U\circ V\subseteq W$ and $x\in U$. Then if $Sxyz$ with $y\in V$, we have $z\in U\circ V$ and hence $z\in W$ which implies $x\in V\rightarrow W$. Now conversely assume that $U\subseteq V\rightarrow W$ and suppose $z\in U\circ V$. The latter implies that there exists $x\in U$ and $y\in V$ such that $Sxyz$. Our assumption that $U\subseteq V\rightarrow W$ together with the fact that $x\in U$ yields $x\in V\rightarrow W$ and since $Sxyz$ with $y\in V$, we have $z\in W$. 

    To see that $\langle\Omega(X);\circ,\Theta\rangle$ is a groupoid, take any $U\in\Omega(X)$ and assume $z\in\Theta\circ U$ so that there exists $x\in\Theta$ and $y\in U$ such that $Sxyz$. By condition 5 of Definition \ref{demg space}, we have $y\leqslant z$. Since $U\in\Omega(X)$, $U$ is open and since $\leqslant$ is the specialization order of $X$, the fact that $y\in U$ yields $z\in U$ and hence $\Theta\circ U\subseteq U$.  Now assume that $z\in U$ and observe that $z\leqslant z$ since $\leqslant$ is reflexive. Then by condition 5 of Definition \ref{demg space}, there exists $x\in\Theta$ with $Sxzz$ and hence $z\in\Theta\circ U$.

    To see that $U\circ(V\cup W)=(U\circ V)\cup(U\circ W)$ for all $U,V,W\in\Omega(A)$, first note that if $z\in U\circ (V\cup W)$, then there exists $x\in U$ and $y\in V\cup W$ such that $Sxyz$. If $y\in V$, then $z\in U\circ V$ so $z\in (U\circ V)\cup(U\circ W)$. If $y\in W$, then $z\in U\circ W$ so $z\in (U\circ V)\cup(U\circ W)$. Hence we have $U\circ(V\cup W)\subseteq(U\circ V)\cup(U\circ W)$. Now assume $z\in (U\circ V)\cup(U\circ W)$ so that $z\in U\circ V$ or $z\in U\circ W$. If $z\in U\circ V$, there exists $x\in U$ and $y\in V$, so $y\in V\cup W$, such that $Sxyz$ so $z\in U\circ (V\cup W)$. If $z\in U\circ W$, there exists $x\in U$ and $y\in W$, so $y\in V\cup W$, such that $Sxyz$ so $z\in U\circ (V\cup W)$. Therefore we have $(U\circ V)\cup(U\circ W)\subseteq U\circ (V\cup W)$. An analogous argument shows that $(U\cup V)\circ W=(U\circ W)\cup(V\circ W)$, which completes the proof.       
\end{proof}

\begin{theorem}
    Every De Morgan groupoid $A$ is isomorphic to $\mathcal{G}_0(\mathcal{F}_0(A))$. 
\end{theorem}
\begin{proof}
    We have seen in the proof of Lemma \ref{lemma 4.7} that $\phi(a)=\{x\in\mathfrak{P}(A):a\in x\}$ is a homomorphism in the sense that
    $\phi(a\cdot b)=\phi(a)\circ\phi(b)$, $\phi(a\rightarrow b)=\phi(a)\rightarrow\phi(b)$, and $\phi(t)=\Theta$. By Applying Theorem \ref{rep theorem}, we obtain the desired isomorphism. 
\end{proof}
\subsection{Algebraic Realization of De Morgan Groupoid Spectral Spaces} In this subsection, we show that every DMGrp-space $X$ is homeomorphic and relationally isomorphic to the spectrum of prime filters $\mathcal{F}_0(\mathcal{G}_0(X))$ of the De Morgan groupoid $\mathcal{G}_0(X)$ of compact open subsets of $X$. 
\begin{theorem}
    Every DMGrp-space $X$ is homeomorphic to $\mathcal{F}_0(\mathcal{G}_0(X))$.  
\end{theorem}
\begin{proof}
    An analogous argument to the one given in the proof of Theorem \ref{homeo} shows that $\psi(x)=\{U\in\Omega(X):x\in U\}$ is a continuous bijection with continuous inverse and hence $\psi$ is a homeomorphism from $X$ to $\mathcal{F}_0(\mathcal{G}_0(X))$.
\end{proof}
The following provides a topological of analogue of the operation defined on filters in Definition \ref{definition 4.3} in the setting of DMGrp-spaces. 
\begin{definition}
If $X$ is a DMGrp-space and $\mathscr{X},\mathscr{Y}\subseteq\mathcal{G}_0(X)$ are filters, let: 
\[\mathscr{X}\odot\mathscr{Y}=\{W\in\Omega(X):U\circ V\subseteq W\hspace{.2cm}\text{for some $U\in\mathscr{X}$ and $V\in\mathscr{Y}$}\}\]
\end{definition}
\begin{corollary}
    Let $X$ be a DMGrp-space. Then: 
\begin{enumerate}
    \item   if $\mathscr{X},\mathscr{Y}\subseteq\mathcal{G}_0(X)$ are filters, then $\mathscr{X}\odot\mathscr{Y}\subseteq\mathcal{G}_0(X)$ is a filter;  
    \item if $\mathscr{X},\mathscr{Y},\mathscr{Z}\subseteq\mathcal{G}_0(X)$ such that $\mathscr{Z}$ is prime and $\mathscr{X}\odot\mathscr{Y}\subseteq\mathscr{Z}$, there there exist prime filters $\mathscr{X}',\mathscr{Y}'\subseteq\mathcal{G}_0(X)$ such that $\mathscr{X}\subseteq\mathscr{X}'$, $\mathscr{Y}\subseteq\mathscr{Y}'$, $\mathscr{X}'\odot\mathscr{Y}\subseteq\mathscr{Z}$, and $\mathscr{X}\odot\mathscr{Y}'\subseteq\mathscr{Z}$.   
\end{enumerate}
 \end{corollary}
\begin{proof}
    The result follows Lemma \ref{lemma 4.3} and Lemma \ref{lemma 4.7}. 
\end{proof}
\begin{theorem}
    Every DMGrp-space $X$ is relationally isomorphic to $\mathcal{F}_0(\mathcal{G}_0(X))$.  
\end{theorem}
\begin{proof}
    First note that for any DMGrp-space $X$, since $\mathcal{G}_0(X)$ is a De Morgan groupoid and $\psi(x)$ is a prime filter in $\mathcal{G}_0(X)$, Definition \ref{prime spectrum1}(4) gives: 
    \[S_{\Omega(X)}\psi(x)\psi(y)\psi(z)\Longleftrightarrow\psi(x)\odot\psi(y)\subseteq\psi(z).\]
    Hence assume $Sxyz$ and let $W\in\psi(x)\odot\psi(y)$ so that $U\circ V\subseteq W$ for some $U\in\psi(x)$ and $V\in\psi(y)$. Thus we have $x\in U$ and $y\in V$ with $Sxyz$ and hence $z\in U\circ V$ so $U\circ V\in\psi(z)$. Therefore we have $\psi(x)\odot\psi(y)\subseteq\psi(z)$ which implies $S_{\Omega(X)}\psi(x)\psi(y)\psi(z)$. Now assume $\overline{S}xyz$ so by condition 4 of Definition \ref{demg space} there exists $U,V\in\Omega(X)$ such that $x\in U$ and $y\in V$ but $z\not\in U\circ V$. Therefore we have $U\in\psi(x)$ and $V\in\psi(y)$ but $U\circ V\not\in\psi(z)$. Clearly $U\circ V\in\psi(x)\odot\psi(y)$ since $U\circ V\subseteq U\circ V$ with $U\in\psi(x)$ and $V\in\psi(y)$ but $U\circ V\not\in\psi(z)$ and hence $\psi(x)\odot\psi(y)\not\subseteq\psi(z)$ and therefore we conclude $\overline{S_{\Omega(X)}}\psi(x)\psi(y)\psi(z)$.  
\end{proof}
Observe that by Theorem \ref{theorem 3.16}, it follows that we have: \[\psi(g(x))=g_{\Omega(X)}(\psi(x))\] for all $x\in X$ where $X$ is any DMGrp-space. 
\begin{theorem}\label{algebraic realization 2}
    In any DMGrp-space $X$, we have $\psi[\Theta]=\Theta_{\Omega(X)}$. 
\end{theorem}
\begin{proof}
    Observe that by condition 3 of Definition \ref{prime spectrum1} together with Theorem \ref{algebraic realization 2} we have  $\Theta_{\Omega(X)}=\{\psi(x):\Theta\in\psi(x)\}$. Hence we have: 
    \[\psi[\Theta]=\{\psi(x):x\in\Theta\}=\{\psi(x):\Theta\in\psi(x)\}=\Theta_{\Omega(X)}\] which follows directly by the definition of $\psi$ and $\Theta_{\Omega(X)}$. 
\end{proof}
\subsection{The Contravariant Functors $\mathcal{F}_F$ and $\mathcal{G}_F$} In this subsection, we introduce the category $\mathbf{DMGrpSpec}$ and show that the category $\mathbf{DMGrp}$ is dually equivalent to $\mathbf{DMGrpSpec}$. 
\begin{definition}
    If $A$ and $A'$ are De Morgan groupoids, a function $h\colon A\to A'$ is a \emph{homomorphism} provided the following conditions are satisfied:  
    \begin{enumerate}
        \item $h(a\wedge b)=h(a)\wedge h(b)$; 
        \item $h(a\vee b)=h(a)\vee h(b)$; 
        \item $h(a\cdot b)=h(a)\cdot h(b)$; 
        \item $h(a\rightarrow b)=h(a)\rightarrow h(b)$; 
        \item $h(t)=t'$; 
        \item $h(0)=0'$; 
        \item $h(1)=1'$. 
    \end{enumerate}
\end{definition}
\begin{definition}
    By $\mathbf{DMGrp}$ we denote the category of De Morgan groupoids and homomorphisms. 
\end{definition}
The following class of functions provide spectral analogous of the morphisms introduced in \cite{urq}. 
\begin{definition}\label{spectral frame morphisms 2}
    Let $X$ and $X'$ be DMGrp-spaces. A function $f\colon X\to X'$ is a \emph{spectral frame morphism} provided the following conditions are satisfied: 
    \begin{enumerate}
        \item $f$ is a spectral map;
        \item $f^{-1}[\Theta']=\Theta$ and $f(g(x))=g(f(x))$;  
        \item $Sxyz\Longrightarrow S'f(x)f(y)f(z)$; 
        \item $S'x'y'f(z)\Longrightarrow\exists x,y\in X(Sxyz\hspace{.1cm}\&\hspace{.1cm}x'\leqslant f(x)\hspace{.1cm}\&\hspace{.1cm}y'\leqslant f(y))$; 
        \item $S'f(x)y'z'\Longrightarrow\exists y,z\in X(Sxyz\hspace{.1cm}\&\hspace{.1cm}y'\leqslant f(y)\hspace{.1cm}\&\hspace{.1cm}f(z)\leqslant z')$. 
    \end{enumerate}
\end{definition}
\begin{proposition}
    Spectral frame morphisms on DMGrp-spaces are closed under composition. 
\end{proposition}
\begin{proof}
     Let $\langle X_1;g_1,S_1,\Theta_1,\tau_1\rangle$, $\langle X_2;g_2,S_2,\Theta_2,\tau_2\rangle$, and $\langle X_3;g_3,S_3,\Theta_3,\tau_3\rangle$  be DMGrp-spaces and $f_1\colon X_1\to X_2$ and $f_2\colon X_2\to X_3$ spectral frame morphisms. Clearly: 
     \[S_1xyz\Longrightarrow S_2f_1(x)f_1(y)f_1(z)\Longrightarrow S_3f_2(f_1(x))f_2(f_1(y))f_2(f_1(z))\] since $f_1$ and $f_1$ are spectral frame morphisms and thus satisfy condition 2 of Definition \ref{spectral frame morphisms 2} and hence $S_1xyz\Longrightarrow S_3f_2\circ f_1(x)f_2\circ f_1(y)f_2\circ f_1(z)$. Moreover: 
     \[(f_2\circ f_1)^{-1}[\Theta_3]=f^{-1}_1\circ f^{-1}_2[\Theta_3]=f^{-1}_1[f^{-1}_2[\Theta_3]]=f^{-1}_1[\Theta_2]=\Theta_1\] so condition 2 is satisfied. For condition 3, assume $S_3x_3y_3f_2\circ f_1(z_1)$ so that we have $S_3x_3y_3f_2(f_1(z_1))$. Then there exists $x_2,y_2\in X_2$ such that $S_2x_2y_2f_1(z_1)$ with $x_3\leqslant f_2(f_1(x_1))$ and $y_3\leqslant f_2(f_1(y_1))$. Since $S_2x_2y_2f_1(z_1)$, there exists $x_1,y_1\in X_1$ such that $S_1x_1y_1z_1$ with $x_2\leqslant f_1(x_1)$ and $y_2\leqslant f_1(y_1)$. Hence we have shown that there exists $x_1,x_2\in X_1$ such that $S_1x_1y_1z_1$ with $x_3\leqslant f_2(f_1(x_1))$ and $y_3\leqslant f_2(f_1(y_1))$ whenever  $S_3x_3y_3f_2\circ f_1(z_1)$, as desired. Condition 4 can be demonstrated analogously and hence by Proposition \ref{composition} we conclude that $f_2\circ f_1\colon X_1\to X_3$ is a spectral frame morphism whenever $f_1\colon X_1\to X_2$ and $f_2\colon X_2\to X_3$ are spectral frame morphisms.    
\end{proof}
\begin{definition}
    By $\mathbf{DMGrpSpec}$ we denote the category of DMGrp-spaces and their associated spectral frame morphisms.  
\end{definition}
\begin{lemma}\label{lemma 4.20}
       Let $A$ and $A'$ be De Morgan groupoids and $h\colon A\to A'$ a homomorphism. Then the map $\mathcal{F}_1\colon\mathcal{F}_0(A')\to\mathcal{F}_0(A)$ defined by $\mathcal{F}_1(h):=h^{-1}$ is a spectral frame morphism from the DMGrp-space $\mathcal{F}_0(A')$ to the DMGrp-space $\mathcal{F}_0(A)$.  
\end{lemma}
\begin{proof}
    By Lemma \ref{homo to spec}, it suffices to demonstrate that conditions 3-5 of Definition \ref{spectral frame morphisms 2} are satisfied. For condition 2, we have:
    \begin{align*}
        \mathcal{F}_1(f)^{-1}[\Theta_{A'}]&=(h^{-1})^{-1}[\Theta_{A'}]
        \\&=\{x\in\mathfrak{P}(A):h^{-1}[x]\in\Theta_{A'}\}
        \\&=\{x\in\mathfrak{P}(A):t\in h^{-1}[x]\}
        \\&=\{x\in\mathfrak{P}(A):h(t)\in x\}
        \\&=\{x\in\mathfrak{P}(A):t\in x\}
        \\&=\Theta_A
    \end{align*}  
    For condition 3, assume $S_{A'}xyz$ so that $x\odot y\subseteq z$ and take $c\in h^{-1}[x]\odot h^{-1}[y]$. By the latter we have $a\cdot b\leq c$ for some $a\in h^{-1}[x]$ and $b\in h^{-1}[y]$ where $h(a)\in x$ and $h(b)\in y$. Since $h$ is a homomorphism and $a\cdot b\leq c$, we have $h(a\cdot b)=h(a)\cdot h(b)\leq h(c)$ and hence $h(c)\in x\odot y$. Since by hypothesis we have $x\odot y\subseteq z$, it follows that $h(c)\in z$ so $c\in h^{-1}[z]$ and thus $S_Ah^{-1}[x]h^{-1}[y]h^{-1}[z]$ so $S_A\mathcal{F}_1(h)[x]\mathcal{F}_1(h)[y]\mathcal{F}_1(h)[z]$.  

    For condition 4, assume $S_{A'}x'y'\mathcal{F}_1(h)[z]$ so that $S_{A'}x'y'h^{-1}[z]$. Then let $[h(x'))$ and $[h(y'))$ be the filters defined by: 
    \[[h(x'))=\{b\in A':h(a)\leq b\hspace{.2cm}\text{for some $a\in x'$}\};\]
     \[[h(y'))=\{b\in A':h(a)\leq b\hspace{.2cm}\text{for some $a\in y'$}\}.\]
     One can easily verify $[h(x'))\odot[h(y'))\subseteq z$ so that  $S_A[h(x'))[h(y'))z'$. By Lemma \ref{lemma 4.3}(2), there exist prime filters $x'_p,y'_p\in\mathfrak{P}(A')$ such that $[h(x'))\subseteq x'_p$ and $[h(y'))\subseteq y'_p$ with $S_{A'}x'_py'_pz'$ where $x'\subseteq h^{-1}[x'_p]=\mathcal{F}_1(h)[x'_p]$ and $y\subseteq h^{-1}[y'_p]=\mathcal{F}_1(h)[y'_p]$. Condition 5 follows via a somewhat analogous argument and exploits the fact that $a\cdot b\leq c$ iff $a\leq b\rightarrow c$ and hence we omit the details.    
\end{proof}
\begin{lemma}\label{lemma 4.21}
  Let $X$ and $X'$ be DMGrp-spaces and $f\colon X\to X'$ a spectral frame morphism. Then the map $\mathcal{G}_1\colon\mathcal{G}_0(X')\to\mathcal{G}_0(X)$ defined by $\mathcal{G}_1(f)=f^{-1}$ is a homomorphism from the De Morgan groupoid $\mathcal{G}_0(X')$ to the De Morgan groupoid $\mathcal{G}_0(X)$.     
\end{lemma}
\begin{proof}
  By Lemma \ref{spec to homo}, it suffices to show    $\mathcal{G}_1(f)[U\circ V]=\mathcal{G}_1(f)[U]\circ\mathcal{G}_1(f)[V]$, $\mathcal{G}_1(f)[U\rightarrow V]=\mathcal{G}_1(f)[U]\rightarrow\mathcal{G}_1(f)[V]$, and $\mathcal{G}_1(f)[\Theta']=\Theta$. For the first equation, assume $z\in\mathcal{G}_1(f)[U\circ V]$ so that $z\in f^{-1}[U\circ V]$ and hence $f(z)\in U\circ V$. Then there exists $x',y'\in X'$ such that $S'x'y'f(z)$ with $x'\in U$ and $y'\in V$. Since $S'x'y'f(z)$, condition 3 of Definition \ref{spectral frame morphisms 2} guarantees the existence of some $x,y\in X$ such that $Sxyz$ with $x'\leqslant f(x)$ and $y'\leqslant f(y)$. Since $x'\in U$ and $x'\leqslant f(x)$ we have $f(x)\in U$ and thus $x\in f^{-1}[U]$. Likewise, since $y'\in V$ with $y'\leqslant f(y)$, we have $f(y)\in V$ so $y\in f^{-1}[V]$. Since $Sxyz$, we obtain $z\in f^{-1}[U]\circ f^{-1}[V]=\mathcal{G}_1(f)[U]\circ\mathcal{G}_1(f)[V]$ so $\mathcal{G}_1(f)[U\circ V]\subseteq\mathcal{G}_1(f)[U]\circ\mathcal{G}_1(f)[V]$. Now take $z\in\mathcal{G}_1(f)[U]\circ\mathcal{G}_1(f)[V]=f^{-1}[U]\circ f^{-1}[V]$ so that there exists $x,y\in X'$ such that $x\in f^{-1}[U]$ and $y\in f^{-1}[V]$ such that $Sxyz$. Then condition 3 of Definition \ref{spectral frame morphisms 2} yields $Sf(x)f(y)f(z)$ and since $f(x)\in U$ and $f(y)\in V$, we have $f(z)\in U\circ V$ so $z\in f^{-1}[U\circ V]=\mathcal{G}_1(f)[U\circ V]$.

For the second equation, assume $x\in\mathcal{G}_1(f)[U\rightarrow V]=f^{-1}[U\rightarrow V]$ so $f(x)\in U\rightarrow V$. Then for all $y'z'\in X'$, if $S'f(x)y'z'$ and $y'\in U$, we have $z'\in U$. Now take any $y,z\in X$ such that $Sxyz$ with $y\in f^{-1}[U]$ so $f(y)\in U$. Since $Sxyz$, condition 3 of Definition \ref{spectral frame morphisms 2} gives $S'f(x)f(y)f(z)$. Since $f(y)\in U$, our hypothesis gives $f(z)\in V$ and hence $x\in f^{-1}[U]\rightarrow f^{-1}[V]=\mathcal{G}_1(f)[U]\rightarrow\mathcal{G}_1(f)[V]$. Now assume $x\in\mathcal{G}_1(f)[U]\rightarrow\mathcal{G}_1(f)[V]=f^{-1}[U]\rightarrow f^{-1}[V]$ and take any $y',z'\in X'$ such that $S'f(x)y'z'$ with $y'\in U$. Since $S'f(x)y'z'$, it follows by condition 5 of Definition \ref{spectral frame morphisms 2} that there exists $y,z\in X$ such that $Sxyz$ with $y'\leqslant f(y)$ and $f(z)\leqslant z'$. Since $y'\in U$ and $y'\leqslant f(y)$, we have $f(y)\in U$ and thus our hypothesis gives $f(z)\in V$. Then since $f(z)\leqslant z'$, we have $z'\in V$ and hence $f(x)\in U\rightarrow V$ so $x\in f^{-1}[U\rightarrow V]=\mathcal{G}_1(f)[U\rightarrow V]$. The third equation follows immediately since $\mathcal{G}_1(f)[\Theta']=f^{-1}[\Theta']=\Theta$ by condition 1 of Definition \ref{spectral frame morphisms 2}. 
\end{proof}
Lemma \ref{lemma 4.7} and Lemma \ref{lemma 4.8} together with Lemma \ref{lemma 4.20} and Lemma \ref{lemma 4.21} imply $\mathcal{F}_F=\langle\mathcal{F}_0,\mathcal{F}_1\rangle\colon\mathbf{DMGrp}\to\mathbf{DMGrpSpec}$ and $\mathcal{G}_F=\langle\mathcal{G}_0,\mathcal{G}_1\rangle\colon\mathbf{DMGrpSpec}\to\mathbf{DMGrp}$ are contravariant functors. Indeed, we have the following.  
\begin{lemma}
    The functors $\mathcal{F}_F=\langle\mathcal{F}_0,\mathcal{F}_1\rangle\colon\mathbf{DMGrp}\to\mathbf{DMGrpSpec}$ and $\mathcal{G}_F=\langle\mathcal{G}_0,\mathcal{G}_1\rangle\colon\mathbf{DMGrpSpec}\to\mathbf{DMGrp}$ are fully faithful. 
\end{lemma}
\begin{proof}
    The proofs run similar to the proofs of Lemma \ref{fully faithful 1} and Lemma \ref{fully faithful 2}.  
\end{proof}

\begin{theorem}\label{duality 2}
    $\mathbf{DMGrp}$ is dually equivalent to $\mathbf{DMGrpSpec}$. 
\end{theorem}
\begin{proof}
    The results collected in this section together with Theorem \ref{duality} give us the desired dual equivalence between $\mathbf{DMGrp}$ and $\mathbf{DMGrpSpec}$. 
\end{proof}

\begin{definition}
    An \emph{S4 De Morgan groupoid} is an algebra $\langle A;\wedge,\vee,-,\cdot,\rightarrow,t,0,1,\nabla\rangle$ satisfying the following conditions: 
    \begin{enumerate}
        \item $\langle A;\wedge,\vee,-,0,1,\nabla\rangle$ is an S4 De Morgan algebra; 
        \item $\langle A;\wedge,\vee,-,\cdot,\rightarrow,t,0,1\rangle$ is a De Morgan groupoid. 
    \end{enumerate}
\end{definition}
Definition \ref{modal de morgan spectral space} together with Definition \ref{demg space} suggest the following construction of the spectral duals of S4 De Morgan groupoids. 
\begin{definition}
       An \emph{S4DMGrp-space} is a relational topological space of the following shape $\langle X;g,S,R,\Theta,\tau\rangle$ such that the reduct $\langle X;g,R,\tau\rangle$ is an S4 De Morgan spectral space and the reduct $\langle X;g,S,\Theta,\tau\rangle$ is a DMGrp-space. Moreover, we call a function $f\colon X\to X'$ between S4DMGrp-spaces a \emph{spectral frame morphism} provided $f$ is a spectral frame morphism between their respective DMGrp-space reducts and satisfies $f^{-1}[\nabla_RU]=\nabla_Rf^{-1}[U]$ for any $U\in\Omega(X')$.  
\end{definition}
Letting $\mathbf{S4DMGrp}$ denote the category of S4 De Morgan groupoids and letting $\mathbf{S4DMGrpSpec}$ denote the category of S4DMGrp-spaces, Theorem \ref{duality} and Theorem \ref{duality 2} immediately yield the following.   
\begin{theorem}
    $\mathbf{S4DMGrp}$ is dually equivalent to $\mathbf{S4DMGrpSpec}$. 
\end{theorem}

\section{Conclusions and future lines of research}
We have obtained duality results for various modal and residuated groupoid expansions of De Morgan algebras using spectral spaces. Our methods adapted those of \cite{bimbo} and \cite{urq} under the isomorphism between $\mathbf{Priest}$ and $\mathbf{Spec}$.

Future lines of investigation may include adapting the spectral dualities obtained in this work to bitopological dualities by means of certain pairwise Stone spaces under the isomorphism established in \cite{bezhanishvili} between $\mathbf{Spec}$ and $\mathbf{PStone}$.  

\par
\vspace{.4cm} 
\noindent \textbf{Acknowledgments.} This work has been funded by a grant from the Programme Johannes Amos Comenius
under the Ministry of Education, Youth and Sports of the Czech Republic,
CZ.02.01.01/00/23$_{-}$025/0008711.

\par
\vspace{.4cm} 
\noindent \textbf{Competing Interests.} The author declares none.

\end{document}